\documentclass[11pt,reqno]{amsart}
\usepackage{fullpage}
\usepackage{amsmath,amsthm,verbatim,amssymb,amsfonts,amscd, graphicx,bbm}
\usepackage{graphics}
\usepackage{mathtools}
\usepackage{mathrsfs}
\usepackage{esint}
\usepackage{color}

\usepackage[colorlinks=true, pdfstartview=FitV, linkcolor=black, citecolor=blue, urlcolor=blue]{hyperref}

\newtheorem{theorem}{Theorem}
\newtheorem{proposition}[theorem]{Proposition}
\newtheorem{lemma}[theorem]{Lemma}
\newtheorem{corollary}[theorem]{Corollary}

\theoremstyle{definition}
\newtheorem{remark}[theorem]{Remark}

\newtheorem{assumption}{Assumption}

\newcommand{\cref}[1]{Corollary~\ref{c.#1}}

\numberwithin{equation}{section}
\numberwithin{theorem}{section}

\newcommand{\Z}{\mathbb{Z}}

\newcommand{\R}{\mathbb{R}}

\newcommand{\bT}{\mathbb{T}}

\newcommand{\cS}{\mathcal{S}}
\newcommand{\cD}{\mathcal{D}}
\newcommand{\cQ}{\mathcal{Q}}
\newcommand{\cP}{\mathcal{P}}

\newcommand{\cK}{\mathcal{K}}
\newcommand{\cR}{\mathcal{R}}

\newcommand{\cA}{\mathcal{A}}

\newcommand{\cE}{\mathcal{E}}
\newcommand{\ep}{\varepsilon}

\newcommand{\pv}{\mathrm{p.v.}}

\renewcommand{\tilde}{\widetilde}

\newcommand{\ol}{\overline}

\newcommand{\eps}{\varepsilon}

\newcommand{\ran}{\mathrm{ran}}

\definecolor{darkgreen}{rgb}{0,0.5,0}
\definecolor{darkblue}{rgb}{0,0,0.7}
\definecolor{darkred}{rgb}{0.9,0.1,0.1}
\definecolor{lightblue}{rgb}{0,0.51,1}

\title{Unified quantitative analysis of the Stokes equations in dilute perforated domains via layer potentials}

\author[W. Jing]{Wenjia Jing}
\address[W. Jing]{Yau Mathematical Sciences Center, Tsinghua University, Beijing 100084 and Yanqi Lake Beijing Institute of Mathematical Sciences and Applications, Beijing 101408, P.R. China}
\email{wjjing@tsinghua.edu.cn}

\author[Y. Lu]{Yong Lu}
\address[Y. Lu]{Department of Mathematics, Nanjing University, Nanjing 210093, P.R. China}
\email{luyong@nju.edu.cn}

\author[C. Prange]{Christophe Prange}
\address[C. Prange]{Cergy Paris Universit\'e, Laboratoire de Math\'ematiques AGM, UMR CNRS 8088, France}
\email{christophe.prange@cyu.fr}

\date{}

\begin{document}

\begin{abstract}
We develop a unified method to obtain the quantitative homogenization of Stokes systems in periodically perforated domains with no-slip boundary conditions on the perforating holes. The main novelty of our paper is a quantitative analysis of the asymptotic behavior of the two-scale cell correctors via periodic Stokes layer potentials. The two-scale cell correctors were introduced and analyzed qualitatively by Allaire in the early 90's \cite{Allaire-3}. Thanks to our layer potential approach, we also provide a novel explanation of the conductivity matrix in Darcy's model, of the Brinkman term in Brinkman's model, and explain the special behavior for $d=2$. Finally, we also prove quantitative homogenization error estimates in various regimes of ratios between the size of the perforating holes and the typical distance between holes. In particular we handle a subtle issue in the dilute Darcy regime related to the non-vanishing of the Darcy velocity on the boundary.

\smallskip


\smallskip

\noindent{\bf Mathematics subject classification (MSC 2020)}: 35B27, 35B40, 35J08, 35Q35, 76D07

\end{abstract}

\maketitle


\section{Introduction}

This paper is devoted to a classical problem in the homogenization of fluids, namely the Stokes problem in a perforated structure $\Omega^{\eps,\eta}$ with $\ep$-periodic distributed holes of size $\ep \eta$, where $\eta$ measures the ratio between the size of the holes and the mutual distance between the holes:
\begin{equation}
\label{eq:hetstokes}
\left\{
\begin{aligned}
&-\Delta u^{\eps,\eta} + \nabla p^{\eps,\eta} = f, &\quad &x \in \Omega^{\eps,\eta},\\
&\nabla \cdot u^{\eps,\eta} = 0, &\quad &x \in \Omega^{\eps,\eta},\\
&u^{\eps,\eta} = 0, &\quad &x \in \partial \Omega^{\eps,\eta},
\end{aligned}
\right.
\end{equation}
where as usual $u^{\eps,\eta}$ is the velocity field and $p^{\eps,\eta}$ the pressure. Here we stress that the domain, the velocity and the pressure all depend on $\eps$ and $\eta$. This dependence on two parameters is at the heart of this work and our goal is to provide a unified construction and analysis of the two-scale correctors that depend on both $\ep$ and $\eta$. Later on, in order to alleviate the notations, we will remove the explicit dependence on $\eta$. 

The homogenization theory concerns the asymptotic behavior of \eqref{eq:hetstokes} in the limit $\eps \to 0$. Note that $\eta^d$ is essentially the volume fraction of the holes and the parameter $\eta$ plays an important role in the homogenization of \eqref{eq:hetstokes}. We say the geometric configuration of the holes is in the \emph{dilute} setting if $\eta \to 0$ as $\eps \to 0$ (e.g., if $\eta = \eps^\alpha$ for some $\alpha > 0$) and is in the \emph{classical} setting if otherwise, i.e., if $\eta$ is of order one. See Figure \ref{fig.porous} for an illustration of dilute holes.

\begin{figure}[t]
\begin{center}
\includegraphics[scale=.6]{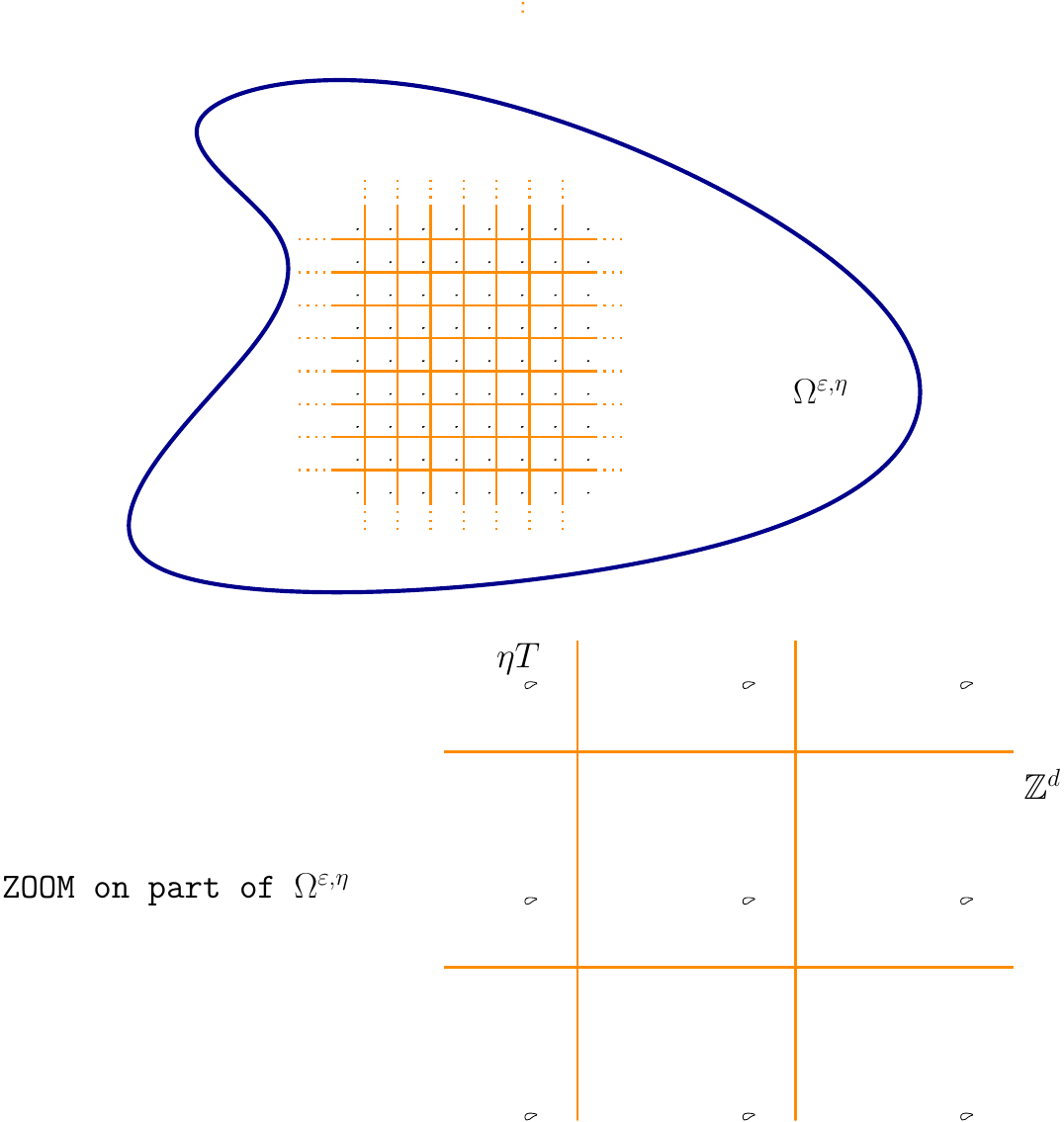}
\caption{The periodically perforated domain}
\label{fig.porous}
\end{center}
\end{figure}

Although this problem has a long history, dating back to the first derivation of Darcy's law by Tartar \cite{Tartar}, there is a recent surge of research activity around system \eqref{eq:hetstokes} that mainly focuses on two issues: (i) optimal quantitative estimates, (ii) the large-scale/uniform regularity theory.  Our work focuses on the first question. We develop a method based on the quantitative analysis of two-scale cell correctors via layer potentials that enables to treat in a unified way all dilute and non-dilute regimes for the parameters $\eps$ and $\eta$. This line of research was pioneered by Allaire in \cite{Allaire-3} for the Stokes system, but it was not until recently that a systematic quantitative study involving layer potentials was carried out by the first author of this paper, first for the Laplace equation \cite{Jing20} and then for the Lam\'e system \cite{MR4172687}; for a qualitative study in the context of the Stokes equations, see \cite{MR4145838}.

For the non-dilute classical case, the first quantitative result for the homogenization of \eqref{eq:hetstokes} was given by Maru\v{s}i\'c-Paloka and  Mikeli\'c  in \cite{MM96}, where $O(\ep^{\frac 16})$ convergence rates were shown in dimension two ($d=2$), and the proof there applies to the stationary Navier-Stokes equations also.  This result was improved recently by Shen in \cite{MR4432951} where sharp convergence rate of order $O(\sqrt \ep)$ was achieved for all $d \geq 2$. Shen also studied the uniform regularity problem for \eqref{eq:hetstokes} in \cite{Shen22}, and established large-scale interior $C^{1, \alpha}$ and interior Lipschitz estimates for the velocity as well as the corresponding estimates for the pressure field. Moreover, uniform $W^{k,p}$ estimates for any $k\geq 1,  \ 1 < p < \infty$ was also established in \cite{Shen22} for the smooth setting, which offered a rigorous proof for the main results claimed by Masmoudi in \cite{Mas04}.  Recently in \cite{BAO24}, Balazi \emph{et.\,al.}\,considered perforated domain with possibly connected solid structure, which is more physical relevant compared with isolated solid obstacles and they obtained sharp $O(\sqrt \ep)$ convergence rate. 

Quantitative homogenization of elliptic equations with high contrast coefficients has been studied intensively as well. In \cite{Shen23-JDE} and furthermore in \cite{SZ23}, quantitative convergence rates, together with the uniform weighted Lipschitz and $W^{1,p}$ estimates were given for scalar equations. See \cite{MR4267502,FJ_highcontrast} for convergence rates and uniform Lipschitz regularity in elliptic systems with high contrast coefficients, and see \cite{WXZ21} for almost-sharp error estimates for linear elasticity systems in periodically perforated domains. In those works, the geometric configuration of the perforating holes or the high contrast inclusions is the non-dilute classical one.

For the dilute case with multi-scale period structure, in a series of papers, Shen \emph{et.\,al.} studied the uniform regularity theory for the scalar Poisson equation in perforated domains with zero Dirichlet boundary there. In \cite{Shen23}, nearly optimal $W^{1,p}, \ 1<p<\infty$ estimates with bounding constants depending explicitly on the small parameters $\ep$ and $\eta$ were established; the results were further improved in \cite{SW23, RS24}.

Our goal in this paper is to present a unified proof of convergence rates of the Stokes system for three different regimes of dilute perforated domain where the limit systems are Darcy's law for the case of relatively large holes, Brinkman's law for the case of critical size holes, and the full domain Stokes equations for relatively small holes. Our proof relies on layer potential analysis of the periodic two-scale cell problem.

\subsection{Criticality: a short state of the art}

The homogenized equation for \eqref{eq:hetstokes} in the classical setting,  i.e., $\eta$ is a fixed number, is Darcy's law. It can be derived using the two-scale expansion method; see e.g., \cite{Keller,Sanchez}. The formal arguments were first made rigorous by Tartar \cite{Tartar}, where a cell problem plays a key role: it is the building block for the construction of oscillating test functions.

The asymptotic behavior of \eqref{eq:hetstokes} in the dilute setting, i.e. $\eta = \eta(\eps)$ tending to zero as $\eps\to 0$, was systematically studied by Allaire \cite{Allaire91-1,Allaire91-2}. These two works were significant generalizations of an earlier work by Cioranescu and Murat \cite{CioMur-1} for the homogenization of Poisson equations in dilutely perforated domains. Allaire introduced the scaling factor 
\begin{equation}
\label{eq:sigeps}
\sigma_\eps := \eps \kappa_\eta^{-1} = \begin{cases}
\eps |\log\eta|^{\frac12}, &\qquad d=2,\\
\eps \eta^{-\frac{d-2}{2}}, &\qquad d\ge 3,
\end{cases}
\qquad\text{and}\qquad
\kappa_\eta := \begin{cases} |\log \eta|^{-\frac{1}{2}}, \qquad &d = 2,\\
\eta^{\frac{d-2}{2}}, \qquad &d\ge 3.
\end{cases}
\end{equation}
and according to the asymptotic behavior of $\sigma_\eps$ as $\eps \to 0$, he identified three different regimes for the homogenization of \eqref{eq:hetstokes} in the dilute setting:  
\begin{itemize}
	\item If $\sigma_\eps \to \sigma_0$ as $\eps \to 0$ and $\sigma_0$ is a positive number comparable to $1$, i.e., $\eta = O(\eps^{2/(d-2)})$ for $d\ge 3$ or $|\log\eta| \sim \frac{1}{\eps^2}$ for $d=2$, we say the configuration is in the \emph{\underline{critical}} setting.  
	\item If $\sigma_\eps \to 0$ as $\eps \to 0$, i.e., $\eta \gg \eps^{2/(d-2)}$ for $d\ge 3$ or $|\log\eta| \gg \frac{1}{\eps^2}$ for $d=2$, we say the configuration is in the \emph{\underline{dilute super-critical}} setting. 
	\item If $\sigma_\eps \to \infty$ as $\eps \to 0$, i.e., $\eta \ll \eps^{2/(d-2)}$ for $d\ge 3$ or $|\log\eta| \ll \frac{1}{\eps^2}$ for $d=2$, we say the configuration is in the \emph{\underline{sub-critical}} setting.
\end{itemize}
Note that the classical setting, where $\eta \sim 1$ and $\sigma_\eps \sim \eps$, is a super-critical one. 

The qualitative homogenization theory in those dilute settings is as follows: (i) For the critical setting, the homogenized system for \eqref{eq:hetstokes} is Brinkman's law where an extra zero-order term appears compared with the Stokes operator. (ii) For the dilute super-critical setting, the homogenized system is the Darcy's law but there are two main differences compared with the non-dilute setting. First, $\sigma_\eps^2$ varies with $\eta$ and determines the $L^2$-scale under which the velocity field converges to the limit; secondly, the permeability matrix in the Darcy's law remains the same, i.e., does not vary with $\eta$, for all dilute super-critical settings and is different from the non-dilute super-critical one. (iii) For the sub-critical setting, the homogenized system is the unperturbed Stokes system in the whole domain without perforation.  

For quantitative convergence rate of periodic homogenization of \eqref{eq:hetstokes} in bounded perforated domain, the critical and sub-critical settings are simpler because the boundary condition on $\partial \Omega$ for the velocity field in the limit system agrees with \eqref{eq:hetstokes}. This is not the case for the super-critical case, in either the classical or the dilute settings; indeed, the velocity field $u$ in Darcy's law only satisfies $u\cdot \nu = 0$ on $\partial \Omega$ and its tangential component can be of order one. It is known that the standard cut-off useful for convergence rate in the homogenization of elliptic equations does not work well for the super-critical settings also because of the incompressibility condition. Construction of boundary layer correctors for $d=2$ was given by Maru\v{s}i\'{c}-Paloka and Mikeli\'{c} \cite{MM96} and they proved $O(\eps^{\frac16})$ convergence rate. Their method worked for certain stationary Navier-Stokes equations but is  restricted to $d=2$ because the tangential velocity corrector is constructed via stream functions. Recently, in \cite{MR4432951} and dealing with the more general case of $d\ge 2$ and non-zero Dirichlet data of order $O(\eps^2)$ on the outer boundary $\partial \Omega$, She proved the sharp $O(\sqrt{\eps})$ homogenization rate for \eqref{eq:hetstokes} by introducing two boundary layer correctors, with properly chosen data on $\partial \Omega$, one tangential and the other normal to $\partial \Omega$. It was observed in \cite{WXZ22} that the so-called radial test functions (see section \ref{subsec.dilute.quant} below) can be constructed so that a boundary cut-off argument goes through for the convergence rate to Darcy's law. In section \ref{subsec.dilute.quant} of this paper, we use this observation for the convergence rate proof for the dilute Darcy's law, using only a cut-off argument.

Very recently after the completion of our work, we noticed that, in \cite{BAO24}, Balazi, Allaire and Omnes studied the homogenization \eqref{eq:hetstokes} for $d\ge 2$ in the more physical relevant case of connected solid structures and obtained $O(\sqrt{\eps})$ convergence rate. Their method utilizes a cut-off argument only because they managed to construct a velocity field supported near $\partial \Omega$ which plays the role of a cut-off of the limit velocity field $u$ but remains incompressible. Their construction generalizes an earlier construction in \cite{MR4714506} for $d=2$. In some sense, the construction is achieved by a remarkable and simple observation that one can insert the usual cut-off inside the $\mathrm{curl}$ operator to simultaneously achieve cut-off and preserve incompressibility. Clearly, we can use this construction in section \ref{subsec.dilute.quant} alternatively. 

\subsection{Allaire's two-scale cell correctors: a unified approach to homogenization in perforated domains}

As in the case of Poisson equations \cite{Jing20} and of Lam\'e systems \cite{MR4172687}, our strategy to establish quantitative homogenization in periodically dilute perforated domains, with Dirichlet type boundary conditions in the holes, is to build two-scale cell correctors.

In this paper, $\eta$ depends in general on $\eps$, and we hence need to study the limit of $(\chi_k,\omega_k)$ as $\eta$ converges to zero with $\eps$. We find it more convenient to work with a rescaled cell problem, defining $\chi^\eta_k(x) = \eta^{d-2}\chi_k(\eta x)$ and $\omega^\eta_k(x) = \eta^{d-1}\omega_k(\eta x)$. These satisfy the following Stokes system in the rescaled torus $\frac1\eta \bT^d$, with a hole $\ol T$ of unit size removed:
\begin{equation}
\label{eq:etacell}
\left\{
\begin{aligned}
&-\Delta \chi^\eta_k + \nabla \omega^\eta_k = \eta^d e_k , &\quad &x \in \eta^{-1}\bT^d \setminus \ol{T},\\
&\nabla \cdot \chi^\eta_k = 0, &\quad &x \in \eta^{-1}\bT^d \setminus \ol{T},\\
&\chi^\eta_k = 0, &\quad &x \in \partial T.
\end{aligned}
\right.
\end{equation}
For each fixed $\eta \in (0,1)$, the classical theory for Stokes system applies to \eqref{eq:etacell}, which asserts that a weak solution $\chi^\eta_k$ is uniquely determined in $\eta^{-1}\bT^d\setminus \ol T$ and $\omega^\eta_k$ is unique up to an additive constant. The correctors $\chi^\eta_k$ and $\omega^\eta_k$ are extended by zero in $T$ so that they are functions on the torus. The precise choice of $\omega^\eta_k$ (outside the hole) is  made clear in \eqref{eq:etachi}. 

Finally, the building blocks for oscillating test functions are given by the further rescaled functions
\begin{equation}
\label{eq:vqeps}
v^\eps_k(x) := \begin{cases}
\chi^\eta_k\left(\frac{x}{\eps \eta} \right)\;&\text{if } d\ge 3\\
\frac{1}{|\log\eta|}\chi^\eta_k\left(\frac{x}{\eps \eta}\right)\;&\text{if } d=2,
\end{cases}
\qquad\text{and} \qquad q^\eps_k (x) := \begin{cases}\frac{1}{\eps \eta} \omega^\eta_k \left(\frac{x}{\eps \eta}\right)\;&\text{if } d\ge 3\\
\frac{1}{\eps\eta|\log \eta|} \omega^\eta_k\left(\frac{x}{\eps \eta}\right)\;&\text{if } d=2.
\end{cases}
\end{equation}
Clearly, the pair $(v^\eps_k,q^\eps_k)$ is $\eps \bT^d$-periodic and satisfies
\begin{equation}
\label{eq:rscell}
\left\{
\begin{aligned}
&-\Delta v^\eps_k + \nabla q^\eps_k = \frac{1}{\sigma^2_\eps} e_k , &\quad &x \in \eps\R^d_f,\\
&\nabla \cdot v^\eps_k = 0, &\quad &x \in \eps\R^d_f,\\
&v^\eps_k = 0, &\quad &x \in \R^d\setminus \eps\R^d_f.
\end{aligned}
\right.
\end{equation}
Note that the holes in $\Omega^{\eps}$ coincide with (part of) those in $\eps \R^d_f$. Hence, given any scalar valued test function $\varphi \in C^\infty_c(\Omega)$, $v^\eps_k \varphi$ is easily checked to be an oscillatory function in $H^1_0(\Omega^\eps)$ and we will use it as the test function for \eqref{eq:hetstokes}. 
Our method for unified quantitative homogenization results relies on the asymptotic analysis for $(v^\eps_k,q^\eps_k)$, in the limit when $\eps \to 0$, see Theorem \ref{eq:vqeps}.

Let us mention that although Allaire did not use the two-scale cell problems \eqref{eq:etacell} and \eqref{eq:rscell} to prove homogenization results in \cite{Allaire91-1,Allaire91-2}, but he studied those problems in \cite{Allaire-3} to establish a continuity of the effective conductivity matrix $M = M_\eta$ in Darcy's law, in the limit when the hole-cell ratio $\eta$ vanishes to zero. In fact, he showed that $(\chi^\eta_k,\omega^\eta_k)$ converge in some proper sense to solutions of a problem which he named the \emph{local} problem: for $d \ge 3$ and for every $k\in\{1,\ldots\, d\}$,
\begin{equation}
\label{eq:locprob}
\left\{
\begin{aligned}
&-\Delta w_k + \nabla q_k = 0, \qquad &\text{in } \R^d\setminus \ol T,\\
&\nabla \cdot w_k = 0, \qquad &\text{in } \R^d\setminus \ol T,\\
&w_k = 0, \qquad &\text{in } \partial T,\\
&w_k \to M^{-1}e_k, \qquad &\text{as } |x| \to \infty.
\end{aligned}
\right.
\end{equation}
Here, $M$ is a symmetric matrix defined by the components $(m_{ik})$ in \eqref{eq:mjk} below. The behavior at infinity of $w_k$, the well-posedness of \eqref{eq:locprob} and the positivity of the matrix $M$ are all established in \cite{Allaire-3}, as parts of the proof of the convergence of $(\chi^\eta_k,\omega^\eta_k)$. Using those results, one can devise a unified proof based on the classical oscillating test function argument for qualitative homogenization that works for all regimes of vanishing $\eta$; see \cite{MR4145838}.

\subsection{Outline of the main results of the paper}

In this paper, we provide a unified quantitative homogenization method for all values of the dilution parameter $\eta$. Our method grounds on quantitative estimates of two-scale correctors via layer potential techniques. Indeed, we quantify the convergence results for the unified correctors $(v^\eta_k,q^\eta_k)$, see the definition \eqref{eq:rscell}, first obtained qualitatively by Allaire \cite{Allaire-3}. 
The advantages of our method are the following. Firstly, the method provides natural correctors $(v^\eta_k,q^\eta_k)$ that enable us to treat all asymptotic regimes, determined by the relative smallness of $\eta$ compared with $\eps$, in a unified manner. Secondly, we obtain explicit representation formulas which lead to quantitative convergence rates for the correctors $(v^\eta_k,q^\eta_k)$. 
Thirdly, by exploring the properties of layer potentials, we provide a novel explanation of the conductivity matrix in Darcy's model, of the Brinkman term in Brinkman's model, and explain the special behavior for $d=2$.

We start with the limit and estimates for the velocity field $v^\eps_k$. Let $M$ be the symmetric positive definite matrix defined by
\begin{equation}
\label{eq:Mdef}
M := \begin{cases}
(A_T)^{-1} \qquad &\text{for } d\ge 3,\\
{4\pi} I \qquad &\text{for } d = 2.
\end{cases}
\end{equation}
Here, for $d\ge 3$, $A_T$ is the matrix defined by \label{eq:Adef} in Lemma \ref{lem:Kmap} using layer potentials. In fact, $M = A_T^{-1}$ is the analog of Newtonian capacitance associated to $T$ in the setting of homogeneous Stokes system.

\begin{theorem}[quantitative analysis of two-scale cell correctors]\label{lem:cellqual}
We assume that Assumption \ref{assump:geo} holds. Let $v^\eps_k$, $k=1,\dots\, d$, be the velocity field defined by \eqref{eq:vqeps} and solving \eqref{eq:rscell}. There exists a universal constant $C> 0$ so that the following results hold. 
\begin{itemize}
\item[(1)] For each $k$ and for all regimes,
\begin{equation}
  \label{eq:vbdd}
\|\nabla v^\eps_k\|_{L^2(\Omega)}  \le C\sigma_\eps^{-1}.
\end{equation}
\item[(2)] In the critical setting and for all $k$,
\begin{equation}
\label{eq:gradvconv}
\nabla v^\eps_k \rightharpoonup  0, \qquad \text{in } L^2(\Omega).
\end{equation}
\item[(3)] In all dilute regimes and for all $k\in\{1,\ldots\, d\}$, with $p = \frac{2d}{d-2}$ for $d \ge 3$ and $p = 2$ for $d=2$, we have 
\begin{equation}
\label{eq:vquant}
\|v^\eps_k - M^{-1}e_k\|_{L^p(\Omega)} \le C \kappa_\eta.
\end{equation}
\end{itemize}
Note that $\sigma_\ep$ and $\kappa_\eta=\eps/\sigma_\eps$ are defined by \eqref{eq:sigeps}.
\end{theorem}

In particular, the result in item (3) above shows, $v^\eps_k$ converges in $L^p_{\rm loc}(\R^d)$ to the constant vector field $M^{-1}e_k$. A qualitative strong convergence result itself can be proved by weak compactness arguments as done in \cite{Allaire-3}. Notice that the rate given here is new and is a byproduct of the layer potential method developed in our paper.

\begin{theorem}[quantitative homogenization]
\label{thm:homerr}
We assume that Assumption \ref{assump:geo} holds. Assume further that, in each regime, the homogenized problem has a solution $u \in W^{2,\infty}(\Omega)$ and $p \in L^2_0(\Omega)$. We can find constants $C$ that depends only on $d$, $\Omega$ and $T$, such that
\begin{itemize}
\item[(1)] In the \emph{\underline{dilute super-critical}} setting,
\begin{equation}
  \label{eq:err_supc}
\big\|\frac{u^\eps}{\sigma^2_\eps} - u\big\|_{L^2(\Omega)}+\|p^\eps - p\|_{L^2(\Omega)} \le C(\kappa_\eta + \sqrt{\sigma_\eps})\|u\|_{W^{2,\infty}}.
\end{equation}
The limit $(u,p)$ is the unique solution to the following Darcy's law:
\begin{equation} 
\label{eq:hpdesupc}
u = M^{-1} (f-\nabla p)\quad\text{and}\quad  \nabla \cdot u = 0\quad\text{in}\quad \Omega, \qquad \text{and}\quad u\cdot n = 0\quad \text{in}\quad \partial \Omega.
\end{equation}
\item[(2)] In the \emph{\underline{critical setting}}, with $\sigma_\eps \to \sigma_0 \in (0,\infty)$, we have
\begin{equation}
\label{eq:err_c}
\big\|u^\eps - \tfrac{\sigma^2_\eps}{\sigma^2_0} v^\eps_k (Mu)^k\big\|_{H^1(\Omega)} + \big\|p^\eps - p -\tfrac{\sigma^2_\eps}{\sigma^2_0} q^\eps_k (Mu)^k\big\|_{L^2} \le C\left(\kappa_\eta + \left|1 - \sigma^2_\eps/\sigma^2_0\right|\right) \|u\|_{W^{2,\infty}}.
\end{equation}
The limit $(u,p)$ is the unique solution to the following Brinkman's law
\begin{equation}
\label{eq:hpdec}
-\Delta u + \nabla p + \frac{M}{\sigma^2_0} u = f \quad \text{and} \quad \nabla \cdot u = 0\quad \text{in}\quad \Omega,  \qquad \text{and}\quad  \quad u = 0\quad \text{in}\quad \partial \Omega.
\end{equation}
\item[(3)] In the \emph{\underline{sub-critical setting}}, we have
\begin{equation}
\label{eq:err_subc}
\|u^\eps - v^\eps_k (Mu)^k\|_{H^1(\Omega)} + \|p^\eps - p\|_{L^2} \le C\left(\kappa_\eta + \sigma^{-2}_\eps\right) \|u\|_{W^{2,\infty}}.
\end{equation}
The limit $(u,p)$ is the unique solution to the unperturbed Stokes problem
\begin{equation}
\label{eq:hpdesubc}
-\Delta u + \nabla p  = f \quad \text{and} \quad \nabla \cdot u = 0\quad \text{in} \quad\Omega, \qquad \text{and}\quad  u = 0\quad \text{in}\quad \partial \Omega.
\end{equation}
\end{itemize}
Note that $\sigma_\ep$ and $\kappa_\eta=\eps/\sigma_\eps$ are defined by \eqref{eq:sigeps}.
\end{theorem}

The convergence rates in Theorem \ref{thm:homerr} were given by Allaire in \cite{Allaire91-1,Allaire91-2} in all three cases, except that in the dilute super-critical setting, the trace of $u$ on the outer boundary $\partial \Omega$ was implicitly assumed to be zero. In Section \ref{sec:app:sup} we establish the convergence rate for the homogenization of \eqref{eq:hetstokes} in the dilute super-critical case in the general setting, i.e. without assuming $u = 0$ on $\partial \Omega$. This turns out to be a very non-trivial task, as is already the case in the classical setting with $\eta = 1$; see \cite{MR4432951} for the discussion of the problem and for the proof in the classical setting.

In the dilute super-critical setting, the additional small parameter $\eta$ only makes the problem even more technical; see in particular the comment below \eqref{eq:sepsk}. Thanks to the estimates of $(v^\eps_k,q^\eps_k)$ achieved in our paper, and in particular the $L^\infty$ estimate \eqref{eq:vqstressbound} for the stress, it is possible to follow and generalize the argument of Shen \cite{MR4432951} to the dilute setting. It turns out, however, an easier approach can be followed using the so-called \emph{radial} cut-off function proposed by Wang, Xu and Zhang \cite{WXZ22}. We follow this second approach. 

\begin{remark}[an explicit representation of Darcy's conductivity matrix and Brinkman's term]
As a byproduct of the analysis of the layer potentials developed in this paper, we obtain explicit formulas for the conductivity matrix in Darcy's regime. 
The basis $\{\phi^*_j\}_{j=1,\dots,d}$, found in \eqref{eq:L2dec}, is closely related to the so-called \emph{local problems} introduced by Allaire in \cite{Allaire91-1}. Indeed, for each $k = 1,\dots,d$, let $(u_k,p_k)$ be $(\cS_T[\phi^*_k]+a^*_k,\cQ_T[\phi^*_k])$ in $T_+ = \R^d\setminus \ol T$, then they satisfy (for $d\ge 3$)
\begin{equation}
\left\{
\begin{aligned}
&\Delta u_k - \nabla p_k = 0, \qquad &\text{in } \R^d\setminus \ol T,\\
&\nabla \cdot u_k = 0, \qquad &\text{in } \R^d\setminus \ol T,\\
&u_k = 0, \qquad &\text{in } \partial T,\\
&u_k \to a^*_k,  \qquad &\text{as } |x| \to \infty. 
\end{aligned}
\right.
\end{equation}
For $d=2$, the last line should be changed to
\begin{equation*}
u_k \sim \Gamma_k(x) + a^*_k, \qquad \text{as } \, |x| \to \infty.
\end{equation*}
Compare the equations above with those in the local problem of \cite[Eqn. (2.4)]{Allaire-3}, or \eqref{eq:locprob} above. We can check that, for $d\ge 3$, the solutions to the local problem are given by the functions $w_k := m_{ki} u_i$ and $q_k := m_{ki} p_i$, where the numbers $m_{ij}$ are the components of $M$. Notice that $M=A_T^{-1}$, where $A_T$ is defined by \eqref{eq:Adef}. We check that
\begin{equation}
\label{eq:mjk}
\begin{aligned}
\int_{\R^d\setminus \ol T} \nabla w_i : \nabla w_k &= m_{is}m_{k\ell} \int_{\R^d\setminus \ol T} \nabla u_s : \nabla u_\ell \\
&= m_{is}m_{k\ell} \int_{\R^d\setminus \ol T} \nabla \cS_T[\phi^*_s] : \nabla \cS_T[\phi^*_\ell] \\
&= -m_{is}m_{k\ell} \int_{\partial T} \cS_T[\phi^*_s] \cdot \frac{\partial \cS_T[\phi^*_\ell]}{\partial \nu}\Big\vert_+  = m_{is} m_{k\ell} a^*_s\cdot e_\ell = m_{ik}.
\end{aligned}
\end{equation}
where $\mathcal S_T$ is the single-layer potential defined by \eqref{eq:cST}. 
This shows that the definition $M := A^{-1}_T$, defined by quantities coming from layer potential theory, agrees with that of Allaire \cite{Allaire-3}, defined using the above `local problems'.
\end{remark}

\subsection{Outline of the paper}

In section \ref{sec.layerpot}, we review the layer potentials for the Stokes system. Section \ref{sec.mapping} focuses on the mapping properties of the layer potentials. In section \ref{sec.twoscale} we introduce and estimate the two-scale cell correctors. Theorem \ref{lem:cellqual} is proved here. Section \ref{sec:app:sup} addresses the error estimates for the homogenization in the dilute super-critical regime, which proves part of Theorem \ref{thm:homerr}. Finally there are four appendices. In appendix \ref{app.a} we review some standard analysis tools, such as the extension and restriction operators. In appendix \ref{app.b}, we carry out basic energy estimates. In appendix \ref{sec.app.qual}, we review the use of the unified two-cell correctors for the qualitative convergence. In appendix \ref{sec.app.quant}, we prove the error estimates in the dilute critical and sub-critical cases, which proves the other cases of Theorem \ref{thm:homerr}.

\subsection{Geometric setup and main assumptions}
\label{sec.geomassu}

We first describe the geometric setup of the perforated domain $\Omega^\ep$; see Figure \ref{fig.porous}. Let $T\subset Q_1$ be a set satisfying the following conditions.
\begin{itemize}
  \item[{\upshape(A1)}] $T$ is an open set and satisfies $B_{1/16} \subset T \subset B_{3/8}$. Roughly speaking, $T$ is relatively centered at the origin and is well separated from $\partial Q_1$. 
	\item[{\upshape(A2)}] $T_+:= \R^d\setminus \ol T$ is connected; $\partial T$ is of class $C^{1,\alpha}$ for some $\alpha \in (0,1)$ and, for simplicity, $T$ is simply connected.
\end{itemize}

Assumption (A2) simplifies the analysis of single- and double-layer potential operators. The fact that $\partial T$ is of class $C^{1,\alpha}$ for some $\alpha \in (0,1)$ implies that $\mathcal K_T-\mathcal K_T^*$ is a compact operator; see Section \ref{sec.mapping}. The fact that $T$ is simply connected implies in particular that that there is just one hole and that $\partial T$ is connected; see Remark \ref{rem.withoutA3}.

Fix a parameter $\eta \in (0,1]$, we define the perforated unit cell to be $Y_{\rm f} = \ol Q_1\setminus (\eta \ol T)$; this set is the building block for $\Omega^\eps$. The subscript $f$ denotes the fluid part of the domain. By the assumptions above, both $Y_{\rm f}$ and the removed set $\eta T$ are connected. In general, we can allow $T$ to have multiple connected components as long as $Y_{\rm f}$ remains connected. For simplicity of the presentation, however, we impose (A2) above. We then take copies of $Y_{\rm f}$, distribute them periodically and glue them together; the resulting set is $\cup_{z\in \Z^d} (z+Y_{\rm f})$ and is denoted by $\R^d_{\rm f}$. Clearly, 
\begin{equation*}
  \R^d_{\rm f} = \R^d \setminus \cup_{z\in \Z^d} \left(z + \eta \ol T\right).
\end{equation*}
We emphasize that $Y_{\rm f}$ and $\R^d_{\rm f}$ depend on $\eta$, as $\eta \ol T$ is the model hole removed in those sets. Given $\eps \in (0,1)$, we rescale $\R^d_{\rm f}$ to $\eps \R^d_{\rm f}$ defined by 
\begin{equation*}
  \eps \R^d_{\rm f} := \R^d\setminus \cup_{z\in \Z^d} \;\eps\left(z + \eta \ol T\right)
\end{equation*}
which is an $\eps$-periodic perforated domain; the array of holes $\cup_{z\in \Z^d} \eps(z+\eta T)$ is denoted by $\eps \R^d_{\rm s}$. Fix an open set $\Omega$ in $\R^d$. We assume that $\Omega$ satisfies the following assumptions.
\begin{itemize}
  \item[{\upshape(A3)}] $\Omega$ is bounded open and simply connected, and the boundary $\partial \Omega$ is of class $C^2$.
\end{itemize}
It would be natural to define the perforated domain $\Omega^\eps$ by $\eps\R^d_{\rm f} \cap \Omega$. However, the small holes in $\eps\R^d_{\rm f}$ may cut $\partial \Omega$ in very irregular ways (possible cusps\ldots). To avoid such situations, we modify the definition by removing (i.e., filling) the holes near $\partial \Omega$. More precisely, for any small number $t>0$, let $\Omega_{(t)}$ denote the set $\{x\in \Omega\,:\, d(x,\partial \Omega) > t\}$. We define
\begin{equation}
  \Omega^{\eps}:=\Omega^{\eps,\eta} := \Omega\setminus \cup_{z\in \mathcal{I}_\eps} \;\eps(z + \eta \ol T),
  \label{eq:Oepsdef}
\end{equation}
where $\mathcal{I}_\eps$ is an index set defined by $\mathcal{I}_\eps := \{z\in \Z^d \,:\, \eps(z+Q_1)\subset \Omega_{(\eps)}\}$. By this construction, the holes in $\Omega^\eps$ never cut $\partial \Omega$; in fact, there is no hole in the layer $\Omega^\eps\cap\Omega_{(\eps)}$ which has a width of order $O(\eps)$. We define $K_\eps := \Omega\setminus \cup_{z\in \mathcal{I}_\eps} \; \eps\left(z+\ol Q_1\right)$; this set contains the layer and serves as a buffer area between $\partial \Omega$ and the $\eps$-cells containing holes.    

To summarize the geometric setup of the perforated domain, we group the assumptions above and invoke throughout the paper the following assumption:

\vspace{-.5cm}
\begin{quote}
\begin{assumption}
  \label{assump:geo}
For each $\eps \in (0,1)$ and $\eta \in (0,1]$, let $T$ and $\Omega$ satisfy the assumptions (A1), (A2) and (A3), and let $\Omega^\eps := \Omega^{\eps,\eta}$ be constructed by \eqref{eq:Oepsdef}.
\end{assumption}
\end{quote}

This is the standing assumption in the paper and, unless explicitly stated otherwise, we always assume Assumption \ref{assump:geo}.

\subsection{Notations} The unit vectors in the Euclidean coordinate system for $\R^d$ are denoted by $e_1,\dots,e_d$. A vector $\xi$ in $\R^d$ is then written as $\xi = (\xi^j)_j$, or $(\xi^1,\xi^2,\dots,\xi^d)$, with $\xi^j$ being the $j$-th coordinate component.
For two matrices $A = (a_{ij})$ and $B = (b_{ij})$ of the same size, their Frobenius inner product is denoted by $A:B$ and defined as $A:B = a_{ij} b_{ij}$. In this definition and in the rest of the paper, the summation convention for repeated indices is assumed unless otherwise stated.
The ball in $\R^d$ with center $x\in \R^d$ and radius $r>0$ is denoted by $B_r(x)$, and by $B_r$ if $x=0$. The unit open cube of $\R^d$ is denoted by $Q_1$ and is the set $(-\frac12,\frac12)^d$. The unit flat torus is denoted by $\bT^d$ and defined by $\R^d/\Z^d$; it can be represented by $\ol Q_1$ with each pair of its opposite faces identified as one. Given a set $E\subset \R^d$ and a positive number $s > 0$, the scaled set $sE$ is defined by $\{sx\,:\, x\in E\}$. The scaled flat torus $s\bT^d$ is defined similarly. 
For a vector field $\xi = (\xi^1,\xi^2,\xi^3) = (\xi^j)$, its derivative $\nabla \xi$ is denoted by the matrix $(\partial_i \xi^j)$. The multiplication $(\nabla \xi)b$, where $b$ is a vector, is computed as $(\partial_i \xi^j)b^j$. 

For a general Hilbert space $H$, the bracket $\langle \cdot,\cdot\rangle_H$ denotes the inner product on this space, and the subscript $H$ is often omitted if it is clear from the context. Let $H'$ be the dual space of $H$, the bracket $\langle\cdot,\cdot\rangle_{H',H}$ denotes the standard dual pairing. For a function $f$ defined in a domain $E$, the notation $\langle f\rangle_E$ stands for the averaged value $|E|^{-1}\int_E f$. 

We mainly work with the standard $L^p(E)$ space and Sobolev spaces $W^{k,p}(E)$ where $E\subset \R^d$ is an open set, namely, $\Omega$ or $\Omega^\eps$. Note, however, we do not distinguish vector fields from scalar ones in the notation for functional spaces. This information should be read from the context. $W^{1,p}(E)$ is simply written as $H^1(E)$, and $H^1_0(E)$ denotes the subspace in which functions have vanishing trace at $\partial E$. 

\section{Layer potentials for the Stokes system}
\label{sec.layerpot}

Here, we state some results on layer potentials associated to the Stokes system. In particular, we choose to study the double-layer potential proposed in \cite{FabKenVer}, which is defined by a particular choice of conormal derivative that is different from the physical one used in Ladyzhenskaya \cite{Ladyzhenskaya}. The main advantage of our choice is the eigenspace of the so-called Neumann-Poincar\'e operator associated to the layer potential operators has a simpler structure.
We also introduce the corresponding periodic layer potentials, and study their mapping properties. Those results on layer potentials will be the main tool for the asymptotic analysis of rescaled problems.

\subsection{Layer potential operators}

The fundamental solution pair for the Stokes system in the free-space $\R^d$ is given by $(\Gamma,\Theta)$, where $\Gamma = (\Gamma_k^j)_{kj}$ is the Kelvin matrix function, $\Theta = (\theta_k)_k$ is the pressure field vector. They are defined by the formulas: 
\begin{equation}
\label{eq:Gammak}
\Gamma_{k}^j(x)= \begin{cases}
\displaystyle -\frac{1}{2d(d-2)\varpi_d} \left[\frac{\delta_{jk}}{|x|^{d-2}} +(d-2)\frac{x^j x^k}{|x|^d}\right] \qquad & d\ge 3,\\
\displaystyle \frac{1}{4\pi}\delta_{jk}\log|x| - \frac{1}{4\pi} \frac{x^j x^k}{|x|^2} \qquad &d = 2,
\end{cases}
\end{equation}
and
\begin{equation}
\label{eq:thetak}
\theta_k(x) = -\frac{x^k}{d\varpi_d |x|^d}, \qquad d\ge 2.
\end{equation}
Here, $\delta_{ij}$ is the Kronecker's $\delta$-symbol, and $\varpi_d$ is the volume of the unit ball in $\R^d$. For each $k=1,2,\dots,d$, the velocity field $\Gamma_k = (\Gamma_k^1,\Gamma_k^2,\dots,\Gamma_k^d)$ formed by the $k$th row of $\Gamma$, together with the pressure field $\theta_k$ formed by the $k$th component of $\Theta$, solve the system 
\begin{equation}
\label{eq:fundmental1}
\Delta \Gamma_k(x) - \nabla \theta_k(x) = \delta_0(x) e_k \quad 
\text{and}\quad \nabla \cdot \Gamma_k(x) = 0, \qquad \text{ in } \R^d.
\end{equation}
Here, $\delta_0(x)$ denotes the Dirac distribution with mass at $0$. Let $\Gamma_\Delta$ be the fundamental solution to the scalar Laplace equation, i.e., solving $-\Delta \Gamma_\Delta = \delta_0$ with proper decay condition at infinity. In fact, $\Gamma_\Delta$ is given by $-\frac1{2\pi}\log|x|$ for $d=2$ and by $\frac{1}{d(d-2)\varpi_d}|x|^{2-d}$ for $d\ge 3$. By direct computations we check that $\theta_k = \partial_k \Gamma_\Delta$. 

Define $\Gamma_k(x,y) := \Gamma_k(x-y)$ and $\theta_k(x,y) := \theta_k(x-y)$, for $x, y \in \R^d$ and $x\ne y$. It is easy to check that 
 \begin{equation}\label{eq:fundamentaly}
\Delta_x \Gamma_k(x,y) - \nabla_x \theta_k(x,y) = \delta_y(x) {e}_k \quad\text{and}\quad
\nabla_x \cdot \Gamma_k(x,y) = 0, \qquad \text{in } \R^d.
\end{equation}
In other words, the family $(\Gamma_k(x,y),\theta_k(x,y))$, $k=1,\dots,d$, is the family of fundamental solutions to the Stokes problem with a singularity at $y$.

Let $\phi \in L^2(\partial T)$. The velocity field of the single-layer potential with momentum $\phi$, associated to $\partial T$, is defined by
\begin{equation}
\label{eq:cST}
\cS_T[\phi](x) = \int_{\partial T} \Gamma_k(x-y)\phi^k(y) \,d\sigma_y,  \qquad x \in \R^d\setminus \partial T.
\end{equation}
Recall that summation over $k$ is assumed here and the right hand sides defines a vector. 
We comment on the notation. We may think of $\Gamma$ as vector field valued tensor, and the operator between $\Gamma$ and $\phi$ yields a vector, i.e., $\Gamma(\phi) = \sum_k \Gamma_k \phi^k$. We may think $\Theta$ as a scalar valued tensor, and $\Theta(\phi)$ is computed as $\sum_{k} \theta_k \phi^k$.
The pressure field of the single-layer potential is defined by
\begin{equation}
\label{eq:cQT}
\cQ_T[\phi](x) = \int_{\partial T} \theta_k(x-y) \phi^k(y) \,d\sigma_y, \qquad x\in \R^d\setminus \partial T.
\end{equation}
Note that the definition of $\cS_T[\phi](x)$ makes sense also for $x \in \partial T$; in contrast, $\cQ_T[\phi](x)$ is not defined for $x\in \partial T$ unless we use take the principal value of the integral in \eqref{eq:cQT}. Direct computation shows that $(\cS_T[\phi],\cQ_T[\phi])$ satisfies the homogeneous Stokes system in $T$ and in $T_+$.

To define the double-layer potential, we need to choose a conormal derivative, associated to a pair $(u,p)$ of velocity field $u$ and pressure field $p$, on any surface contained in the domain of fluid flow. In this paper, we find it is most convenient to use the one proposed in \cite{FabKenVer}, defined by
\begin{equation}
\label{eq:conorm}
\frac{\partial[u,p]}{\partial \nu} := -pN + (\nabla u(x))N = (-p N^i + N^j\partial_j u^i).
\end{equation}
Here, $N$ is the unit outer normal field along a chosen surface in the region of fluid flow. We emphasize that this choice of conormal derivative is different from the physical one, which reads
\begin{equation}
\label{eq:conormF}
\frac{\partial[u,q]}{\partial \rho} := -pN + 2\mathbb{D}(u) N, \qquad \mathbb{D}(u) = \frac12(\nabla u + (\nabla u)^{\rm T}).
\end{equation}
In fluid mechanics, $\frac{\partial[u,p]}{\partial \rho}$ is the stress tensor imposed by the fluid flow on the surface. In comparison, $\frac{\partial [u,p]}{\partial \nu}$ is an artificial modification of the physically meaningful conormal derivative. The reason for our choice is the following: it yields a Neumann-Poincar\'e operator whose eigenspace associated to the eigenvalue $\frac12$ is simpler. Using integration by parts, we verify that, for any pair $(u,p)$ of velocity and pressure field, and for any velocity field $v$, with sufficient regularity, the following Green's identity holds:
\begin{equation}
\label{eq:greenid1}
 \int_{\partial T} \frac{\partial[u,q]}{\partial \nu_y} \cdot v  = \int_{T} (\Delta u - \nabla q) \cdot v - q(\nabla\cdot v) +  \int_T \nabla u : \nabla v.
\end{equation}
Applying this identity to two pairs $(u,q)$ and $(v,p)$, with $u$ and $v$ being solenoidal, we also obtain the second Green's identity:
\begin{equation}
\label{eq:greenid2}
\int_{T} (\Delta u - \nabla q) \cdot v - (\Delta v - \nabla p) \cdot u 
= \int_{\partial T} \left[\frac{\partial [u,q]}{\partial \nu}\cdot v  - \frac{\partial [v,p]}{\partial \nu} \cdot u \right].
\end{equation}

In view of those Green identities and in view of \eqref{eq:fundamentaly}, we define the velocity field of the double-layer potential, associated to the momentum $\phi \in L^2(\partial T)$ on $\partial T$, by
\begin{equation}
\label{eq:cDT}
\cD_T[\phi](x) = \int_{\partial T} \frac{\partial[\Gamma_k,\theta_k](y-x)}{\partial \nu_y}  \phi^k(y)\,dy, \qquad x\in \R^d \setminus \partial T.
\end{equation}
Note that, $(\Gamma_k(y-x),\theta_k(y-x))$ satisfies \eqref{eq:fundamentaly} which has a singularity at the point $x$. In the computation of $\partial/\partial \nu_y$, derivatives are taken with respect to the $y$-variable. For the associated pressure field, define, $x \in \R^d\setminus \partial T$,
\begin{equation}
\label{eq:cPT}
\begin{aligned}
\cP_T[\phi](x) &= \int_{\partial T} N^\ell_y \frac{\partial \theta_k(x-y)}{\partial y_\ell} \phi^k(y)\, dy\\
&= \frac{1}{d\varpi_d} \int_{\partial T} N^\ell_y \left(\frac{-\delta_{k\ell}}{|x-y|^d} + \frac{d(x-y)^k(x-y)^\ell}{|x-y|^{d+2}}\right)\phi^k(y)\, dy.
\end{aligned}
\end{equation}
We verify that $(\cD_T[\phi],\cP_T[\phi])$ solves the homogeneous Stokes system in $T$ and in $T_+$, i.e.,
\begin{equation*}
\left(\Delta_x \cD_T[\phi] - \nabla_x \cP_T[\phi]\right)(x) = 0, \qquad x\in \R^d\setminus \partial T.
\end{equation*}
 If we try to extend the definitions of $\cS_T$, $\cQ_T$, $\cD_T$ and $\cP_T$ to $x \in \partial T$, we can check that $\cS_T[\phi]$ is well defined, but not the other operators. In fact, $\cQ_T$ involves an integral kernel of order $-(d-1)$, $\cP_T$ has an integral kernel of order $-d$, even on smooth domains. For $\cD_T[\phi]$, the integral kernel is:
\begin{equation}
\label{eq:Kij}
\begin{aligned}
 K_{jk}(x,y) = & -\theta_k(y-x) N_y^j + N^\ell_y \partial_{y_\ell} (\Gamma_k^j(y-x))\\
 = & \frac{ N_y^k(x-y)^j-(x-y)^kN_y^j}{2d\varpi_d|x-y|^d}- \frac{1}{2\varpi_d} \frac{\langle N_y, x-y\rangle (x-y)^j (x-y)^k}{|x-y|^{d+2}}\\
 &\qquad - \frac{1}{2d\varpi_d} \frac{\langle N_y,x-y\rangle \delta_{jk}}{|x-y|^d}.
\end{aligned}
\end{equation}
For $x,y \in \partial T$ with $x\ne y$, the function $K_{jk}$ can be viewed as an integral kernel of order $-(d-1)$. In fact, for $\partial T$ of class $C^{1,\alpha}$, using
\begin{equation}
\label{eq:C1alpha}
|\langle N_y, x-y\rangle| \le C|x-y|^{1+\alpha}, \qquad |N_x - N_y| \le C|x-y|^\alpha,
\end{equation}
we see that the last two terms in the expression of $K_{jk}$ contributes to a weakly singular integral operator. The first term, however, results in a singular integral operator. By the standard but deep theory of Coifman, McIntosh and Meyer \cite{CMcIM82},  
\begin{equation}
\label{eq:cKT}
(\cK_T[\phi])^k(x) := \pv \int_{\partial T} K_{jk}(x,y) \phi^j(y) dy, \qquad x \in \partial T
\end{equation}
defines a bounded linear operator from $L^2(\partial T)$ to $L^2(\partial T)$, and $\cK_T$ (and, equally often, its adjoint operator) is referred to as the Neumann-Poincar\'e operator.

\subsection{Jump relations} The layer-potentials are very useful in solving boundary value problems for the homogeneous Stokes system, because one can derive trace formulas for them. For a function (vector-valued or scalar-valued) $F$ that is defined in $T_- := T$ and/or in $T_+$, we denote its trace from $T_-$ on $\partial T$ by $F\rvert_-$, and denote its trace from $T_+$ on $\partial T$ by $F\rvert_+$: 
\begin{equation}
  F\rvert_\pm (x) = \lim_{t\to 0+} F(x \pm t N_x), \qquad x \in \partial T.
\end{equation}
Here, $N_x$ is the unit outer normal vector at $x$. 

For the double-layer velocity field $\cD_T[\phi]$, we have
\begin{equation}
\label{eq:trcD}
 \cD_T [\phi]\big\rvert_\pm(x) = \lim_{t\to 0+} \cD_T [\phi](x \pm tN_x) = \left(\mp \frac12 I + \cK_T \right)[\phi](x),\qquad x\in\partial T.
\end{equation}

For the derivatives of the single-layer velocity field, i.e. $\nabla \cS_T[\phi]$, and for the single-layer pressure field $\cQ_T[\phi]$, we have 
\begin{equation}
\label{eq:traceu}
\begin{aligned}
\partial_j (\cS_T[\phi])^i \big\rvert_\pm(x) &= \pm \frac{1}{2} \left\{N_x^j \phi^i(x) - N^j_x N^i_x \langle N_x, \phi(x)\rangle\right\} + \pv \int_{\partial T} \partial_j \Gamma^i_k(x-y) \phi,\\
\cQ_T[\phi]\big\rvert_\pm (x) &= \mp \frac12 \langle N_x,\phi(x)\rangle + \pv \int_{\partial T} \theta_k(x-y) \phi^k(y).
\end{aligned}
\end{equation}
Combining the terms above according to our choice \eqref{eq:conorm} of conormal derivative, we check that 
\begin{equation}
\label{eq:trcS}
\frac{\partial[\cS_T[\phi],\cQ_T[\phi]]}{\partial \nu} \Big\rvert_\pm(x) = \left(\pm \frac12 I + \cK_T^* \right)[\phi](x).
\end{equation}
Here $\cK_T^*$ is defined by
\begin{equation}
\label{eq:cKstarT}
(\cK^*_T[\phi])^i(x) = \pv \int_{\partial T} K^*_{ik}(x,y)\phi^k(y) dy,
\end{equation}
and the integration kernel $K^*_{ik}$ is given by
\begin{equation*}
K^*_{ik}(x,y) = -\theta_k(x-y)N^i_x + N^j_x\partial_j \Gamma^i_k(x-y).
\end{equation*}
We refer to \cite{FabKenVer} for those formulas.

Direct computation shows that $K^*_{ik}(x,y) = K_{ki}(y,x)$, so $\cK^*_T$ is the adjoint operator of $\cK_T$, and $\cK^*_T$ is a non-compact bounded linear operator on $L^2(\partial T)$. We then compute that
\begin{equation}\label{e.KK*}
\begin{aligned}
(K^*_{ik}-K_{ik})(x,y) = &\frac{1}{2\varpi_d} \frac{\langle N_x+N_y, x-y\rangle (x-y)^i (x-y)^k}{|x-y|^{d+2}} + \frac{1}{2d\varpi_d} \frac{\langle N_x+N_y,x-y\rangle \delta_{ik}}{|x-y|^d} \\
&+ \frac{(N_x - N_y)^k(x-y)^i- (N_x-N_y)^i(x-y)^k}{2d\varpi_d |x-y|^d}.
\end{aligned}
\end{equation}
In view of \eqref{eq:C1alpha}, the integral operator $\cK^*_T - \cK_T$ is compact on $L^2(\partial T)$. This fact plays an important role in the study of those layer potential operators in \cite{FabKenVer}.

\section{Mapping properties of the layer potentials}
\label{sec.mapping}
In this section, we study the mapping properties of some operators related to the single- and double-layer potentials, and establish layer potential representation formulas for the cell problem \eqref{eq:etacell} that play an essential role in the homogenization of \eqref{eq:hetstokes}. Since the cell problem is a Dirichlet boundary value problem, we focus on the double-layer potential representation of the solution, in particular, on the mapping properties of the corresponding Neumann-Poincar\'e operator $\cK_T$ and its adjoint $\cK_T^*$. 

\subsection{Free-space Stokes potentials}

In the sequel, $L^2_0(\partial T)$ denotes the subspace of $L^2(\partial T)$ consisting of mean-zero functions on $\partial T$. We have the following important results:

\begin{lemma}\label{lem:Kmap} Let $d\ge 2$. Then the following results hold.
\begin{itemize}
\item[(1)] The operator $-\frac12 I + \cK^*_T$ has range in $L^2_0(\partial T)$, and the restricted operator $-\frac12 I+ \cK^*_T : L^2_0(\partial T) \to L^2_0(\partial T)$ is a bijection.

\item[(2)] The operators $-\frac12 I + \cK^*_T$ and $-\frac12 I + \cK_T$ both have closed ranges in $L^2(\partial T)$, and both have kernels of dimension $d$.

\item[(3)] The space $\ker(-\frac12 I + \cK_T)$ has basis $\{e_\ell \,:\, 1\le \ell\le d\}$. There is a basis $\phi_1,\dots,\phi_d$ for the space $\ker(-\frac12 I + \cK^*_T)$ and there are $d$ vectors $a_1,\dots,a_d$ in $\R^d$, such that 
\begin{equation}
\label{eq:phij}
-\cS_T[\phi_j](x) = a_j \; \text{in } \ol T, \qquad \int_{\partial T} \phi_j =  e_j.
\end{equation}
Moreover, the matrix $A_T$ defined by
\begin{equation}
\label{eq:Adef}
A_T = \begin{bmatrix}
a_1 & a_2 & \cdots & a_d
\end{bmatrix},
\end{equation}
where the $a_j$'s are written as column vectors, is symmetric for $d\ge 2$; for $d\ge 3$, $A_T$ is also positive-definite.

\item[(4)] The following direct-sum decompositions hold:
\begin{equation}
\label{eq:L2dec}
\begin{aligned}
L^2(\partial T) &= \ran (-\tfrac12 I + \cK_T^*) \oplus \ker(-\tfrac12 I + \cK_T^*),\\
L^2(\partial T) &= \ran (-\tfrac12 I + \cK_T) \oplus \ker(-\tfrac12 I + \cK_T).
\end{aligned}
\end{equation}
\end{itemize}
\end{lemma}

We emphasize that identities in \eqref{eq:L2dec} are, in general, not orthogonal decompositions, and that $A_T$ defined by \eqref{eq:Adef} can be degenerate for $d = 2$ (see Remark \ref{rem:ATpositive} below). 
The first two items of the lemma above are proved by Fabes, Kenig and Verchota \cite{FabKenVer} in the setting that $\partial T$ is Lipschitz. 
Because the non-physical conormal derivative \eqref{eq:conorm} is used in the double-layer potential \eqref{eq:cDT}, the resulted Neumann-Poincar\'e operators $\cK_T$ and its adjoint $\cK^*_T$ are not compact even when $\partial T$ is smooth. Hence, one cannot use the usual Fredholm theory directly for $\cK_T$ and $\cK^*_T$. Nevertheless, the Fredholm property stated in the first two items above were proved in \cite{FabKenVer}, in fact even for Lipschitz $\partial T$. The authors there established the closedness of the ranges of $-\frac12 I + \cK^*_T$ and $-\frac12 I + \cK_T$ using Rellich formulas associated to the Stokes system with conormal derivative \eqref{eq:conorm}, and utilized the fact that $\cK_T - \cK_T^*$ is compact. We refer to \cite[Lemmas 2.1, 2.2, 2.4 and 2.5]{FabKenVer} for the details.

\begin{proof}[Proof of Lemma \ref{lem:Kmap}]
  It remains to prove the last two items. 

  \noindent{\underline{Proof of item (3)}}. \emph{Characterization of $\ker(-\frac12 I + \cK_T)$}: Consider the unit constant vector fields $e_\ell$, $\ell \in \{1,\cdots,d\}$. In view of Green's identity \eqref{eq:greenid1}, they satisfy
\begin{equation*}
\mathcal{D}_T [e_\ell] (x) = \begin{cases}
0, \qquad & x \in \ol{T}^c,\\
e_\ell, \qquad & x \in T.
\end{cases}
\end{equation*}
Using Green's identity \eqref{eq:greenid1} again with $(u,p) = (\Gamma_k,\theta_k)$ and $v = e_\ell$ in the domain $T\setminus \overline{B}_\delta(x)$, and by explicit computations of the resulting boundary integral over $T\cap \partial B_\delta(x)$ and sending $\delta \to 0$, we also have
\begin{equation}
\label{eq:cKT12}
\cK_T[e_\ell](x) = \frac12 e_\ell, \qquad x\in \partial T, \quad \forall \ell = 1,\dots,d.
\end{equation}
Hence, $\mathrm{span}\{e_\ell \,:\, 1\le \ell\le d\}$ is a subspace of $\ker(-\frac12 I + \cK_T)$ with dimension $d$. As a result, the latter must be spanned by $e_\ell$, $\ell=1,\dots,d$. 

\emph{Characterization of $\ker(-\frac12 I + \cK^*_T)$}: we use an argument in \cite[Theorem 2.26]{AmmKan}. Let $\ker(-\frac12 I + \cK^*_T) \otimes \R^d$ denote the product Hilbert space equipped with the standard inner product
\begin{equation*}
\langle (\varphi,a), (\psi,b)\rangle = \langle \varphi,\psi\rangle_{L^2(\partial T)} + a\cdot b, \qquad \text{for } \, \varphi,\psi \in \ker(-\tfrac12 I + \cK^*_T), \, a,b\in \R^d. 
\end{equation*}
This is a Hilbert space with dimension $2d$. The product space $\ker(-\frac12 I + \cK_
T)\otimes \R^d$ and the inner product on it are defined in the same way. Consider the following mapping:
\begin{equation}
\label{eq:KerKer}
\begin{aligned}
\cA_T \quad  : \quad \ker(-\tfrac12 I + \cK^*_T) \otimes \R^d \; &\to\; \ker (-\tfrac12 I + \cK_T) \otimes \R^d \\
(\phi,a) \; &\mapsto \; (\cS_T[\phi] + a, \int_{\partial T} \phi).
\end{aligned}
\end{equation}
Here, $\cS_T[\phi]$ is understood as the trace of the single-layer velocity field on $\partial T$. The mapping is well defined because $\phi$ being in $\ker(-\frac12 I + \cK^*_T)$ implies that $\cS_T[\phi]$ is a constant in $\ol T$. Indeed, let $(u,p) := (\cS_T[\phi], \cQ_T[\phi])$ in $T$ and in $T_+$. Then, due to the jump relation \eqref{eq:trcS} and to the fact that $(-\frac12 I + \cK^*_T)[\phi] = 0$, 
\begin{equation*}
\frac{\partial [u,p]}{\partial \nu}\Big\rvert_- = 0, \qquad \frac{\partial[u,p]}{\partial \nu}\Big\rvert_+  = \phi.
\end{equation*}
By applying Green's identity in $T$, we conclude that $u$ is a constant in $\ol T$.

We show next that \emph{$\cA_T$ is a bijection}.
Since $\cA_T$ is a linear mapping between two Hilbert spaces of the same finite dimension, it suffices to check the injectivity of $\cA_T$. To this end, suppose $\phi \in \ker(-\frac12 I + \cK^*_T)$ and $a \in \R^d$ satisfy
\begin{equation*}
\cS_T[\phi] + a = 0, \qquad \text{and}\qquad \int_{\partial T} \phi = 0.
\end{equation*}
Let $(u,p) = (\cS_T[\phi],\cQ_T[\phi])$ in $T$ and in $T_+$. Then in addition to $u$ being a constant in $T$, we also have sufficient decay of $u$ in $T_+$ at infinity (for $d=2$ we use the fact $\phi \in L^2_0$). This allows us to apply Green's identity in $T_+$ for all $d \ge 2$; it follows that $u$ is a constant in $\R^d$ and, by the jump relation, $\phi = 0$. By the formula above, $a = 0$ as well. This verifies the bijectivity of $\cA_T$.

As an application, for each $j = 1,\dots,d$, there is a unique pair $(\phi_j, a_j)$ so that $\mathcal{A}_T(\phi_j,a_j) = (0,e_j)$. This is precisely the solution pair to \eqref{eq:phij}. Moreover, $\phi_j$'s are clearly independent and form a basis for $\ker(-\frac12 I + \cK^*_T)$, via \eqref{eq:trcS}. This provides a basis for $\ker(-\frac12 I + \cK_T^*)$.

\emph{Symmetry and positivity of the matrix $A_T$}: we first check that the $(i,j)$-element of the matrix $A_T$ defined by \eqref{eq:Adef} is given by
\begin{equation*}
(a_j)^i = - e_i \cdot \cS_T[\phi_j] = -\left(\int_{\partial T} \phi_i \right) \cdot \cS_T[\phi_j] = -\langle \phi_i, \cS_T[\phi_j]\rangle_{L^2(\partial T)}.
\end{equation*}
Since the restriction of $\cS_T$ on $\partial T$ is clearly a self-adjoint operator in $L^2(\partial T)$, we get $(a_j)^i = (a_i)^j$. That is, $A_T$ is symmetric.
For the positivity of $A_T$, let $d\ge 3$. Fix any $c \in \R^d$ and compute
\begin{equation*}
\langle c, A_T c\rangle = - c^i \langle \phi_i, \cS_T[\phi_j]\rangle_{L^2(\partial T)} c^j = -\langle \phi, \cS_T[\phi]\rangle_{L^2(\partial T)}, \qquad \text{with}\; \phi = c^j \phi_j.
\end{equation*}
Clearly, $\phi \in \ker(-\frac12 I + \cK^*_T)$, and $\int_{\partial T} \phi = c$. Let $(u,p)$ be $(\cS_T[\phi],\cQ_T[\phi])$ in $T$ and in $T_+$. We deduce, as before, $u$ is a constant in $\ol T$ and $\phi = \frac{\partial [u,p]}{\partial \nu}\rvert_+$ on $\partial T$. Since $d\ge 3$, $u$ has sufficient decay at infinity and we check that
\begin{equation*}
-\langle \phi, \cS_T[\phi]\rangle_{L^2(\partial T)} = -\int_{\partial T} \frac{\partial[u,p]}{\partial \nu}\Big\rvert_+ \cdot u =\int_{T_+} |\nabla u|^2.
\end{equation*}
The quantity above is non-negative, and it equals zero if and only if $u$ is a constant in $\R^d$, which would imply $\phi = 0$ and $c = 0$. As a result, $A_T$ is positive definite for $d\ge 3$. 

\medskip

\noindent{\underline{Proof of item (4)}}. We have seen that $\ran(-\frac12 I + \cK^*_T)$ is $L^2_0(\partial T)$, whose codimension agrees with $\dim \ker(-\frac12 I + \cK^*_
T)$. For the first identity in \eqref{eq:L2dec}, it suffices to show the right hand side is a direct sum. Suppose $\phi = (-\frac12 I + \cK^*_T)[g]$ and $(-\frac12 I + \cK^*_T)[\phi] = 0$, then as an element in the range of $-\frac12 I + \cK^*_T$, we have $\phi \in L^2_0$. Now we can argue as before for the pair $(u,p) = (\cS_T[\phi],\cQ_T[\phi])$, defined in $T$ and in $T_+$. In particular, we can apply Green's identity in $T_+$ for all $d\ge 2$ and eventually deduce $\phi = 0$. This establishes the first identity in \eqref{eq:L2dec}. The second identity then follows from the \emph{orthogonal} decompositions
\begin{equation*}
 L^2(\partial T) = \ran(-\tfrac12 I + \cK_T) \oplus \ker(-\tfrac12 I + \cK^*_T) =  \ran(-\tfrac12 I + \cK^*_T) \oplus \ker(-\tfrac12 I + \cK_T). 
\end{equation*} 
The proof of the lemma is hence complete.
\end{proof}

\begin{remark}
The argument at the end of the proof of item three cannot go through for $d = 2$ because, if $c\ne 0$ there is not enough decay for $u$ and $p$ defined in the process, and we cannot apply the Green's identity in $T_+$.
When $d\ge 3$, the argument in the proof can be used to show that, the operator $\cS_T[\phi] : \ker(-\frac12 I + \cK^*_T) \to \ker(-\frac12 I + \cK_T)$ is a bijection. This, again, is not true for $d = 2$.
\end{remark}

\begin{remark}\label{rem:ATpositive}
For $d=2$, the logarithmic term in the fundamental solution \eqref{eq:Gammak} yields a particular rescaling property for $\cS_T$. As a result, $A_{T}$ can be degenerate as we show below.

Let $r > 0$; we note that, for all $j,k=1,2,\dots,d$,
\begin{equation*}
\Gamma^j_k\big(\frac{x}{r}\big) = \Gamma^j_k(x) -\frac1{4\pi} (\log r) \delta_{jk}.
\end{equation*}
By the definition of single-layer potentials, we get
\begin{equation*}
\begin{aligned}
\cS_T[\psi](x) &= \int_{\partial T} \Gamma_k\big(\frac{x-y}{r}\big)\psi^k(y)dy + \frac1{4\pi} (\log r) \int_{\partial T} \psi(y)\\
&= r\int_{\partial (\frac1r T)} \Gamma_k\big(\frac{x}{r}-y'\big)\psi^k(ry')dy' + \frac1{4\pi} (\log r) \int_{\partial T} \psi(y)\\
&= r\cS_{\frac1r T}[\psi(r\cdot)]\big(\frac{x}{r}\big) + \frac1{4\pi}(\log r) \int_{\partial T} \psi(y)dy.
\end{aligned}
\end{equation*}
We also check, using the homogeneity (of degree $-1$) of the integration kernel of $\cK^*_T$, that
\begin{equation*}
\cK^*_{\frac1r T}[\psi(r\cdot)]\big(\frac{x}{r}\big) = \cK^*_T[\psi](x).
\end{equation*}
Now consider the basis $\{\phi_j\}$ for $\ker(-\frac12 I + \cK^*_T)$, let $\phi_{j,r}(z) := r \phi_j(rz)$, for $z \in \partial (\frac1r T)$. We check that $\phi_{j,r} \in L^2(\frac1r\partial T)$ and that
\begin{equation*}
\begin{aligned}
&\int_{\frac1r \partial T} \phi_{j,r} = \int_{\partial T} \phi_j = e_j,\\
&(-\tfrac12 I + \cK^*_{\frac1r T})[\phi_{j,r}](z) = r(-\tfrac12 I + \cK^*_T)[\phi_j](rz) = 0, \qquad z \in \partial(\frac{1}{r} T).
\end{aligned}
\end{equation*}
As a result, $\{\phi_{j,r}\}$ are the basis of $\ker(-\frac12 I + \cK^*_{\frac1r T})$ determined by \eqref{eq:phij}. In view of the definition \eqref{eq:Adef}, we check that $A_T$ has the following rescaling property:
\begin{equation}
A_{r T} = A_T + \frac1{4\pi} (\log r) I. 
\end{equation}
We see that given $T$, there always exist one or two $r > 0$ such that $A_{rT}$ is degenerate, and there are at most two such $r$. As long as the homogenization of \eqref{eq:hetstokes} in the dilute setting is considered, we can always assume $\det A_T \ne 0$ because, if this condition fails, we can choose $r_0$ slightly less than one so that $A_{r_0 T}$ is invertible. Then replace $T$ by $r_0T$, and replace $\eta$ by $\eta/r_0$, the asymptotic analysis problem for \eqref{eq:hetstokes} as $\eps \to 0$ is unchanged. 
\end{remark}

The decomposition of $L^2(\partial T)$ established in \eqref{eq:L2dec} can be carried out explicitly, following the construction of the basis $\{\phi_j\}$ of $\ker(-\frac12 I + \cK^*_T)$. 

\begin{proposition}
\label{prop:L2dec}
Let $\Pi_0$ be the projection operator from $L^2(\partial T)$ on $\ker(-\tfrac12 I + \cK_T)$ and $\Pi_1$ be the projection operator from $L^2(\partial T)$ on $\ran(-\tfrac12 I + \cK_T)$ defined according to \eqref{eq:L2dec}. More precisely, given $\psi \in L^2(\partial T)$, $(\Pi_0[\psi],\Pi_1[\psi])$ is the unique pair such that
\begin{equation*}
\psi = \Pi_0[\psi] + \Pi_1[\psi], \qquad \Pi_0[\psi] \in \ker(-\tfrac12 I + \cK_T)\ \text{and} \  \Pi_1[\psi] \in \ran(-\tfrac12 I + \cK_T).
\end{equation*}
Then, $\Pi_0$ is given by the following formula:
\begin{equation}
\label{eq:proj0}
(\Pi_0[\psi])^k = \langle \phi_k, \psi\rangle_{L^2(\partial T)}, \qquad k = 1,\dots,d.
\end{equation}
\end{proposition}

\begin{remark}Consider the fundamental solutions $\Gamma_\ell(x)\rvert_{\partial T}$, $\ell = 1,\dots,d$, defined by \eqref{eq:Gammak} and apply the decomposition formula to them. We find, for $d\ge 3$, 
\begin{equation}
(\Pi_0[-\Gamma_\ell])^k = -\int_{\partial T} \Gamma_\ell(y) \cdot \phi_k(y) = -(\cS_T[\phi_k])^\ell(0) = (a_k)^\ell.
\end{equation}
In other words, the columns of the matrix $A_T$ are precisely the projections $\Pi_0[-\Gamma_\ell]$, $\ell=1,\dots,d$.
\end{remark}

\begin{remark}\label{rem.withoutA3} In this paper we assumed that $T$ is connected, i.e., there is only one hole inside the unit cell. This simplifies matter for our approach that is based on layer potential theory. Namely, it allows us to use results from \cite{FabKenVer} directly (where a single hole was assumed), and, moreover, $\ker(-\frac12 I + \cK_T)$ and $\ker(-\frac12 I + \cK_T^*)$ have rather simple structure. Our approach, however, can be generalized to treat multiple holes, e.g., $T$ has $N$ connected components and $\bT^d\setminus \ol T$ is still connected. Then the kernel of $-\frac12 I + \cK$ and its adjoint will have dimension $N\times d$, and the concept of $A_T$ should be replaced by $N\times N$ matrix with entries that are themselves $d\times d$ matrices. The permeability matrix $M$ should be defined by the inverse of a proper contraction of $A_T$.
\end{remark}

\subsection{Periodic Stokes potentials}
\label{sec.perpot}

To find a representation formula for the solution of the cell problem \eqref{eq:etacell}, we define periodic versions of the layer potential operators for the Stokes system. The fundamental solutions of the Stokes system on the unit torus $\bT^d$, denoted by $(G_k,p_k)$, for each $k=1,\dots,d$, satisfy the equation
\begin{equation*}
\left\{
\begin{aligned}
&\Delta G_k(x) - \nabla P_k(x) = (\delta_0(x) - 1)e_k,\\
&\nabla \cdot G_k(x) = 0,
\end{aligned}
\right. 
\qquad \text{in } \bT^d.
\end{equation*}
Here, $G_k$ and $P_k$ are understood as functions defined on the flat torus $\bT^d$, or, equivalently, as periodic functions over $\R^d$. The existence and uniqueness (up to additive constants) for $(G_k,P_k)$ is a standard result, and explicit formulas can be obtained in terms of Fourier series. To fix the constants, we assume that $G_k$ and $P_k$ are mean zero on $\bT^d$. It is also clear that we can write
\begin{equation}\label{e.defRk}
G_k(x) = \Gamma_k(x) + R_k(x), \quad\text{and}\quad P_k(x) = \theta_k(x) + h_k(x),
\end{equation}
where $R_k$ is a vector field on the unit cube $Q_1$ and $h_k$ is a scalar field on $Q_1$. Moreover, by subtracting the equations satisfied by $(\Gamma_k,\theta_k)$ from those satisfied by $(G_k,P_k)$, we can derive equations satisfied by $(R_k,h_k)$, in $Q_1$, and deduce that $R_k$ and $h_k$ are smooth functions in the closed cube $\ol Q_1$, although $\nabla R_k$ and $h_k$ (if extended periodically) are not continuous across $\partial Q_1$. 

Taking divergence on the first equation for $(G_k,P_k)$, we find that
\begin{equation*}
-\Delta P_k = \partial_k (\delta_0-1) \qquad \text{in } \bT^d.
\end{equation*}
This implies that, up to an additive constant, $P_k$ is given by
\begin{equation}
\label{eq:PkGdel}
P_k = -\partial_k G_\Delta,
\end{equation}
where $G_\Delta$ is the fundamental solution to the Laplace equation on the unit flat torus. This can be viewed as an analog of the formula \eqref{eq:thetak}.

On the rescaled torus $\eta^{-1}\bT^d$, the fundamental solutions $(G^\eta_k,P^\eta_k)$ are given by
\begin{equation*}
G^\eta_k(x) = \eta^{d-2}G_k(\eta x), \qquad P^\eta_k(x) = \eta^{d-1}P_k(\eta x), \qquad x\in \eta^{-1}\bT^d.
\end{equation*}
and they satisfy
\begin{equation}
\label{eq:Geta}
\left\{
\begin{aligned}
&\Delta G^\eta_k(x) - \nabla P^\eta_k(x) = (\delta_0(x) - \eta^d)e_k,\\
&\nabla \cdot G^\eta_k(x) = 0,
\end{aligned}
\right.
\qquad \text{in } \eta^{-1}\bT^d.
\end{equation}
For $\eta$ small, those functions are perturbations of the free-space fundamental solutions:
\begin{equation}
\label{eq:Gperturb}
G^\eta_k(x) = \Gamma_k(x) + \eta^{d-2} R_k(\eta x), \quad P^\eta_k(x) = \theta_k(x) + \eta^{d-1} h_k(\eta x).
\end{equation}
Note also, slightly abusing the notation, the functions $G^\eta_k(x,y) := G^\eta_k(x-y)$ and $P^\eta_k(x,y):=P^\eta_k(x-y)$, for $x,y\in \eta^{-1}\bT^d$, satisfy
\begin{equation*}
  \Delta_x G(x,y) - \nabla_x P^\eta(x,y) = (\delta_y(x) - \eta^d)e_k, \qquad \text{in } \; \eta^{-1}\bT^d.
\end{equation*}

Now we introduce the periodic layer potentials for Stokes systems on the torus $\eta^{-1}\bT^d$ and associated to the set $\ol T$. We always assume that $T$ satisfies the conditions in Assumption \ref{assump:geo}.

The periodic single-layer velocity field and pressure fields are defined by: 
\begin{equation}
\label{eq:cSeta}
\begin{aligned}
&\cS^\eta_T[\phi](x) := \int_{\partial T} G^\eta_k(x-y) \phi^k(y) d\sigma_y, \\
&\cQ^\eta_T[\phi](x) := \int_{\partial T} P^\eta_k(x-y) \phi^k(y) d\sigma_y,
\end{aligned}
\qquad x \in \eta^{-1} \bT^d \setminus \partial T.
\end{equation}
It is straightforward to check that
\begin{equation*}
\Delta \cS^\eta_T[\phi] - \nabla \cQ^\eta_T[\phi] = -\eta^d \int_{\partial T} \phi, \qquad  x\in \eta^{-1}\bT^d \setminus \partial T.
\end{equation*}
Hence, the pair $(\cS^\eta_T[\phi],\cQ^\eta_T[\phi])$ solves the homogeneous Stokes system away from $\partial T$ if and only if $\phi$ is mean-zero on $\partial T$.

The periodic double-layer velocity and pressure fields are defined by:
\begin{equation}
\label{eq:cDeta}
\begin{aligned}
&\cD^\eta_T[\phi](x) := \int_{\partial T} \frac{\partial [G^\eta_k, P^\eta_k](y-x)}{\partial \nu_y} \phi^k(y) d\sigma_y,\\
&\cP^\eta_T[\phi](x) := \int_{\partial T} N_y^i\partial_{y_i} [P_k^\eta(y-x)] \phi^k(y) d\sigma_y,
\end{aligned}
\qquad x \in \eta^{-1}\bT^d \setminus \partial T.
\end{equation}
It can be checked that the pair $(\cD^\eta[\phi], \cP^\eta[\phi])$ solve the homogeneous Stokes system away from $\partial T$ for all $\phi$ in $L^2(\partial T)$. In view of the perturbative formula \eqref{eq:Gperturb}, the operators $\cD_T^\eta$, $\cS^\eta_T$, $\cP_T^\eta$, $\cQ^\eta_T$, etc., are perturbations of the corresponding operators defined in the free-space. In particular, we can check from direct computations that
\begin{equation}
\label{eq:cDeJump}
\cD^{\eta}_T[\phi](x) = \cD_T[\phi](x) + \eta^{d-1}\mathcal{R}^\eta_T[\phi](x), \qquad \forall x \in \eta^{-1}\bT^d\setminus \partial T,
\end{equation}
where $\mathcal{R}_T^\eta[\phi]$ is defined as 
\begin{equation*}
\mathcal{R}^\eta_T[\phi](x) := \int_{\partial T}[-h_k(\eta(y-x))N_y + N^\ell_y \partial_{y_\ell} R_k(\eta (y-x))] \phi^k(y) \,d\sigma_y.
\end{equation*}
Since $h$ and $\nabla R_k$ are bounded functions over $Q_1$, the integral operator $\mathcal{R}^\eta_T$ has an integration kernel that can be uniformly bounded in $\eta$ (in fact, smallness can be explored). It follows that the periodic double-layer velocity field has the following jump condition across $\partial T$:
\begin{equation}
\cD^\eta_T[\phi]\Big\rvert_\pm(x) = \left(\mp \frac12 I + \cK^\eta_T\right)[\phi](x), \qquad x \in \partial T.
\end{equation}
Here $\cK^\eta_T$ is simply given by
\begin{equation*}
\cK^\eta_T := \cK_T + \eta^{d-1}\cR^\eta_T,
\end{equation*}
and $\cR^\eta_T$ is defined as before but with $x \in \partial T$. Because $\cR^\eta_T$ has a uniformly bounded integration kernel, $\cR^\eta_T$ maps $L^2(\partial T)$ to $L^\infty(\partial T)$, and $\cR^\eta_T$ is compact as an operator on $L^2(\partial T)$.

In a similar way, the conormal derivative of the periodic single-layer potentials satisfies the following jump formula:
\begin{equation}
\frac{\partial(\cS^\eta_T,\cQ_T^\eta)[\phi]}{\partial \nu}\Big\rvert_\pm (x) = \left(\pm \frac12 I + \cK^{\eta,*}_T\right)[\phi](x).
\end{equation}
Here, $\cK^{\eta,*}_T$ is the adjoint operator of $\cK^\eta_T$, and it can be written as $\cK^*_T + \eta^{d-1}(\cR^\eta_T)^*$. 

Because $\cK^\eta_T$ (respectively, $\cK^{\eta,*}_T$) is a small (for small $\eta$) and compact perturbation of $\cK_T$ (respectively, $\cK^*_T$), many mapping properties of $-\frac12 I + \cK^\eta_T$ (and $-\frac12 I + \cK^{\eta,*}_T$) follow from those of $-\frac12 I + \cK_T$; in particular, the following results are useful. 

\begin{lemma}
\label{lem:Kper}
Let $d\ge 2$. The following statements hold.
\begin{itemize}
	\item[(1)] For each $e_k$, $k=1,\dots,d$, the identity $(-\frac12 I + \cK^\eta_T)[e_k] = -\eta^d|T|e_k$ holds.
	\item[(2)] The kernel of $-\frac12 I + \cK^\eta_T$ is $\{0\}$.
\item[(3)] The operator $-\frac12 I + \cK^\eta_T : L^2(\partial T) \to L^2(\partial T)$ is invertible.
\end{itemize}
\end{lemma}

\begin{proof}\ \\
\noindent\underline{First item.}\\
We take $x\in T$ and apply the Green's identity in $T$, we obtain
\begin{equation*}
  \cD_T^\eta[e_k](x) = \int_{\partial T} \frac{\partial[G^\eta_\ell,P^\eta_\ell](y-x)}{\partial \nu_y} \delta_{k\ell} d\sigma_y = \delta_{\ell k} \int_T (\delta(y-x) - \eta^d)e_\ell dy = (1-\eta^d|T|)e_k.
\end{equation*}
Sending $x$ to the boundary $\partial T$, we get the first result. As a consequence, for any $\phi \in \ker(-\frac12 I + \cK^{\eta,*}_T)$, the following holds:
\begin{equation*}
\int_{\partial T} e_k \cdot \phi = -\frac{1}{\eta^d |T|} \int_{\partial T} (-\tfrac12 I + \cK^\eta_T)[e_k] \cdot \phi =  -\frac{1}{\eta^d |T|} \int_{\partial T} e_k \cdot(-\tfrac12 I + \cK^{\eta,*}_T)[\phi] = 0.
\end{equation*}
In other words, we have $\ker(-\frac12 I + \cK^{\eta,*}_T) \subseteq L^2_0(\partial T)$.

\smallskip

\noindent\underline{Second item.}\\
Since $\cK^\eta_T$ (respectively, $\cK^{\eta,*}_T$) is a compact perturbation of $\cK_T$ (respectively, $\cK^*_T$) and $-\frac12 I + \cK_T$ is Fredholm (as in Lemma \ref{lem:Kmap}), we conclude that the operator $-\frac12 I + \cK^\eta_T$ and its adjoint operator are Fredholm. It suffices to check that $\ker(-\frac12 I + \cK^{\eta,*}_T)$ contains only $0$. Take any element $\phi$ in this kernel, we have seen that $\phi \in L^2_0(\partial T)$. Let $(u,p)$ be the pair $(\cS^\eta_T[\phi]\vert, \cQ^\eta_T[\phi]\rvert)$, then they satisfy the homogeneous Stokes system in $T$ and in $\eta^{-1}\bT^d\setminus \ol T$. By the trace formulas, we have
\begin{equation*}
\frac{\partial[u,p]}{\partial \nu}\Big\rvert_{-} = \left(-\frac12 I + \cK^{\eta,*}_T\right)[\phi] = 0, \qquad \phi = \frac{\partial[u,p]}{\partial \nu}\Big\rvert_{+}, \qquad \text{in } \partial T.
\end{equation*}
Applying Green's identity, first in $T$ and then in $\eta^{-1}\bT^d\setminus \ol T$, we conclude that $u$ is a constant in $\eta^{-1}\bT^d$, $p = 0$ in $T$, and $p$ is a constant in $\eta^{-1}\bT^d\setminus \ol T$. Denote this constant by $p_+$. Then $\phi = -p_+ N$ on $\partial T$. We observe from the definition of the single-layer pressure that
\begin{equation*}
  \int_{\eta^{-1}\bT^d} \cQ^\eta_T[\phi](x) dx = \int_{\partial T}\int_{\eta^{-1}\bT^d} P^\eta_k(x-y)dx  \phi(y) dy = 0.
\end{equation*}
In other words, $p = \cQ^\eta_T[\phi]$ averages to zero in $\eta^{-1}\bT^d$. Since $p\rvert_T = 0$, it follows that $p_+ = 0$ and, hence, $\phi = 0$. This verifies $\ker(-\frac12 I + \cK^{\eta,*}_T) = \{0\}$ and proves item two. 

\smallskip

\noindent\underline{Third item.}\\
It follows immediately since we have shown that $-\frac12 I + \cK^\eta_T$ is Fredholm and is injective on $L^2(\partial T)$. The proof of the lemma is now complete.
\end{proof}
\section{The two-scale cell problems}
\label{sec.twoscale}

The goal of this section is to prove Theorem \ref{lem:cellqual}. We first obtain a representation formula for the two-scale correctors by inverting the periodic layer potentials defined in Section \ref{sec.perpot}.

\subsection{Layer potential representation formulas}

Our strategy for solving the two-scale corrector problem \eqref{eq:etacell} is based on the comparison with the system \eqref{eq:Geta}. Indeed, we seek for $(\Phi^\eta_k,\psi^\eta_k)$ so that the following formulas hold:
\begin{equation}
\label{eq:chiform}
\chi^\eta_k = G^\eta_k + \Phi^\eta_k, \qquad \omega^\eta_k = P^\eta_k + \psi^\eta_k, \qquad \text{in }\, \eta^{-1}\bT^d\setminus \ol T.
\end{equation}
The pair $(\Phi^\eta_k, \psi^\eta_k)$ then solves the problem
\begin{equation}
\label{eq:perext}
\left\{
\begin{aligned}
&\Delta \Phi^\eta_k - \nabla \psi^\eta_k = 0, &\quad &x \in \eta^{-1}\bT^d \setminus \ol{T},\\
&\nabla \cdot \Phi^\eta_k = 0, &\quad &x \in \eta^{-1}\bT^d \setminus \ol{T},\\
&\Phi^\eta_k = -G^\eta_k(\cdot,0), &\quad &x \in \partial T,;
\end{aligned}
\right.
\end{equation}
notice that we use assumption (A1), namely that $0\in T$. 
This is a Dirichlet boundary value problem on the perforated torus $\eta^{-1}\bT^d\setminus \ol{T}$, with smooth Dirichlet data $-G^\eta_k$ given on $\partial T$. The following lemma provides explicit formulas for the solution.

\begin{lemma}
\label{lem:chiomega}
Let $d\ge 2$ and suppose Assumption \ref{assump:geo} holds. Then, for each $k = 1,\dots,d$, the unique solution $(\chi^\eta_k,\omega^\eta_k)$ (for $\omega^\eta_k$ the uniqueness is up to an additive constant) to \eqref{eq:etacell} is given by
\begin{equation}
\label{eq:etachi}
\begin{aligned}
  &\chi^\eta_k(x) = G^\eta_k(x) + A_T e_k + \cD^\eta_T[\tilde g](x) + r^\eta_k, \\
  &\omega^\eta_k(x) = P^\eta_k(x) + \cP^\eta_T[\tilde g](x),
\end{aligned}
\end{equation}
for $x \in \eta^{-1}\bT^d\setminus \ol T$. Here, $r^\eta_k$ is a constant vector whose norm is of order $O(\eta^{d-2})$; $t^\eta_k$ is a real number of order $\eta^{d-1}$; $g$ is a function in $L^2(\partial T)$ with mean value $\langle g\rangle_{\partial T}$ on $\partial T$, and $\tilde g : = g-\langle g\rangle_{\partial T}$ is the mean-zero part of $g$. Moreover, there exists some universal constant $C > 0$ so that
\begin{equation}
\label{eq:gbdd}
|\langle g\rangle_{\partial T}| \le C\eta^{-1}, \qquad \|g-\langle g\rangle_{\partial T}\|_{L^2(\partial T)} \le C.
\end{equation}
\end{lemma}

\begin{proof}\ \\
\noindent\underline{Step 1: a decomposition of the boundary data.}\\
In view of the jump condition \eqref{eq:cDeJump} and the invertibility established in Lemma \ref{lem:Kper}, the equations \eqref{eq:perext} can be solved by periodic double-layer potentials. However, in order to prove the estimates, we first decompose the boundary data $-G^\eta_k\rvert_{\partial T}$ according to \eqref{eq:L2dec}. Recall the projection operators in Proposition \ref{prop:L2dec}; set
\begin{equation}
\label{eq:Gbddec}
-G^\eta_k = c^\eta_k + h^\eta_k, \qquad c^\eta_k = \Pi_0 [-G^\eta_k\rvert_{\partial T}], \; h^\eta_k = \Pi_1 [-G^\eta_k\rvert_{\partial T}].
\end{equation}
Then according to \eqref{eq:proj0} and the perturbative formula \eqref{eq:Gperturb}, we have
\begin{equation*}
\begin{aligned}
&c^\eta_k = \Pi_0[-G^\eta_k\rvert_{\partial T}] = \Pi_0[-\Gamma_k\rvert_{\partial T}] + {\tilde r}^{\,\eta}_k = A_T e_k + {\tilde r}^{\,\eta}_{k},\\
&({\tilde r}^{\,\eta}_k)^\ell = -\eta^{d-2} \langle \phi_\ell, R_k(\eta x)\rangle_{L^2(\partial T)},
 \end{aligned}
\end{equation*}
where $R_k$ is defined by \eqref{e.defRk}. In view of the uniform boundedness of $R_k$, we check that $|\tilde r^{\,\eta}_k| \le C\eta^{d-2}$ for all $k$, for some universal constant $C$.

\smallskip

\noindent\underline{Step 2: layer-potential formulation for the two-scale cell problem.}\\ The previous step suggests that the velocity field of \eqref{eq:perext} can be solved by $\Phi^\eta_k(y) = A_T e_k + \tilde r^{\,\eta}_k + \cD^\eta_T[g](y)$ and $\psi^\eta_k(y) = \cP^\eta_T[g](y)$. Indeed, by the trace formula \eqref{eq:cDeJump}, we only need to find $g \in L^2(\partial T)$ so that
\begin{equation*}
(-\tfrac12 I + \cK^\eta_T)[g] = h^\eta_k, \qquad \text{in } L^2(\partial T).
\end{equation*}
We further write $g$ as the sum of its mean value $\langle g\rangle_{\partial T}$ and its fluctuation part $\tilde g= g - \langle g\rangle$, the lattering being in $L^2_0(\partial T)$. In view of Lemma \ref{lem:Kper}, we rewrite the equation above as
\begin{equation*}
(-\tfrac12 I + \cK_T + \eta^{d-1}\cR^\eta_T)[\tilde g] -\eta^d|T|\langle g\rangle = h^\eta_k, \qquad \text{in } L^2(\partial T).
\end{equation*}

\noindent\underline{Step 3: solution and estimates for $g$.}\\
Let the projection operators $\Pi_1$ and $\Pi_0$ act on the equation above. We get
\begin{equation}
\label{eq:getaform}
\begin{aligned}
&(-\tfrac12 I + \cK_T + \eta^{d-1}\Pi_1\cR^\eta_T)[\tilde g] = h^\eta_k,\\
&\eta^{d-1}\Pi_0\cR^\eta_T[\tilde g] - \eta^d|T|\langle g\rangle = 0.
\end{aligned}
\end{equation}
Since the operator $\cR^\eta_T$ is uniformly bounded independent of $\eta$, for $\eta$ sufficiently small, the operator on the left hand side in the first line is a small perturbation of $-\frac12 I + \cK_T$, which is invertible on $L^2_0$. By the standard perturbation theory, we get
\begin{equation*}
\tilde g = (-\tfrac12 I + \cK_T + \eta^{d-1}\Pi_1\cR^\eta_T)^{-1}[h^\eta_k], \qquad \text{and} \quad \|\tilde g\|_{L^2(\partial T)} \le C\|h^\eta_k\|_{L^2(\partial T)}.
\end{equation*}
The estimate for $\langle g\rangle$ then follows from the second equation of \eqref{eq:getaform}. For $x\in \eta^{-1}\bT^d\setminus \ol T$, the contribution $\cD^\eta_T[\langle g\rangle]$ is the constant $-\eta^d|T|\langle g\rangle$ which is of order $O(\eta^{d-1})$. Add this constant to $\tilde r^\eta_\eta$ to define the constant $r^\eta_k$, which is still of order $O(\eta^{d-2})$. 

For $\omega^\eta_k$, we note that $\psi^\eta_k = \cP^\eta_T[g] = \cP^\eta_T[\tilde g]$ and $(\Phi^\eta_k,\psi^\eta_k)$ solves the problem \eqref{eq:perext}. Hence, the second line of \eqref{eq:etachi} holds. 
The proof is now complete.
\end{proof}

The formulas \eqref{eq:chiform} specify the values of $(\chi^\eta_k,\omega^\eta_k)$ in the perforated torus $\eta^{-1}\bT^d\setminus \ol T$, and we extend their values by zero in $T$. 

\begin{corollary}
\label{coro:chi}
Suppose that the assumptions of Lemma \ref{lem:chiomega} hold. Let $\langle \chi^\eta_k\rangle$ and $\langle \omega^\eta_k \rangle$ denote the mean values of $\chi^\eta_k$ and $\omega^\eta_k$ in $\frac1\eta \bT^d$. Then there exists a universal constant $C$, such that the inequalities
\begin{equation}
\label{eq:chiunif}
\left\| \nabla \chi^\eta_k \right\|_{L^2(\frac1\eta \bT^d)} + \left| \langle \chi^\eta_k\rangle \right| \le \begin{cases}
C \;&\text{if } d\ge 3\\
C|\log\eta|^{\frac12} \;&\text{if } d = 2,
\end{cases}\qquad
\left|\langle \omega^\eta_k\rangle\right| \le C\eta^d,
\end{equation}
\begin{equation}
\label{eq:omeavg}
\text{and}\quad \|\omega^\eta_k - \langle \omega^\eta_k\rangle \|_{L^2(\eta^{-1}\bT^d)} \le \begin{cases} C, \quad &\text{if } d\ge 3,\\
C|\log \eta|^{\frac12}, \quad &\text{if } d =2.
\end{cases}
\end{equation}
hold for all $k=1,\dots,d$. Moreover, when $d\ge 3$, we also have
\begin{equation}
\label{eq:chiavg}
\left|\langle\chi^\eta_k\rangle - A_T e_k \right| \le C\eta^{d-2}.
\end{equation}
\end{corollary}

\begin{proof}\ \\
\noindent\underline{Step 1: the proofs for \eqref{eq:chiunif}.}\\
Integrate $\chi^\eta_k$ against the first equation in \eqref{eq:etacell} and apply the Poincar\'e inequality. This results in the estimates for $\chi^\eta_k$ in \eqref{eq:chiunif}. To control the mean value of $\omega^\eta_k$, in view of the formula \eqref{eq:etachi}, it suffices to show the integrals of $P^\eta_k$ and of $\cP^\eta_T[\tilde g]$ over the torous $\eta^{-1}\bT^d$ are bounded uniform in $\eta$. In view of the oddness of $P^\eta_k$ and the fact that $\ol T\subset B_1$,
\begin{equation*}
\int_{\eta^{-1}\bT^d\setminus \ol T} P^\eta_k(y) dy = \int_{B_1\setminus \ol T} P^\eta_k(y) dy \le C,
\end{equation*}
where $C$ depends only on $d$. For the integral of $\cP^\eta_T[\tilde g]$, we compute relying on \eqref{eq:PkGdel}
\begin{equation*}
\int_{\eta^{-1}\bT^d\setminus \ol T} \cP^\eta_T[\tilde{g}](x) dx = \int_{\eta^{-1}\bT^d\setminus \ol T} \int_{\partial T} N_y^i \partial_{y_i}\left(-(\partial_{k} G^\eta_\Delta)(y-x)\right)\tilde{g}^k(y) d\sigma_y dx.
\end{equation*}  
In the integral on the right hand side above, we can rewrite $\partial_kG^\eta_\Delta(y-x)$ as $-\partial_{x_k}(G^\eta_\Delta(y-x))$ and then take $\partial_x$ outside of the integral over $\partial T$. We get
\begin{equation*}
-\int_{\eta^{-1}\bT^d\setminus \ol T} \cP^\eta_T[\tilde g ](x) dx = -\int_{\eta^{-1}\bT^d\setminus \ol T} \partial_{x_k} \cD^\eta_{T,\Delta}[\tilde g^k](x) dx,
\end{equation*}
where $\cD^\eta_{T,\Delta}[\tilde g^k]$ is the double-layer potential associated to the Lapace equation in $\eta^{-1}\bT^d$. Finally, by the divergence theorem, the integral above becomes
\begin{equation*}
\int_{\partial T} N^k_y \cD^\eta_{T,\Delta}[\tilde g^k]\big\rvert_+(y) d\sigma_y = \int_{\partial T} N^k_y \left(-\frac12 I + \cK^\eta_{T,\Delta}\right)[\tilde g^k](y) \,d\sigma_y.
\end{equation*}
Here, we used the trace formula for the double-layer potential $\cD^\eta_{T,\Delta}$. For $\eta \in (0,1)$, the operator $-\frac12 I + \cK^\eta_{T,\Delta}$ is uniformly bounded in $\mathcal{L}(L^2(\partial T))$ (more properties are established in \cite{Jing20}), and by \eqref{eq:gbdd}, the above is uniformly bounded in $\eta$. Hence, the second inequality of \eqref{eq:chiunif} holds.

\smallskip

\noindent\underline{Step 2: the proof of \eqref{eq:chiavg}.}\\
For the average of $\chi^\eta_k$, we compute from \eqref{eq:etachi} and \eqref{eq:Gperturb}
\begin{equation*}
\langle \chi^\eta_k\rangle - A_T e_k = \frac{1}{|\eta^{-1}\bT^d|} \int_{\eta^{-1}\bT^d} \Gamma_k(x) + \cD^\eta_T[\tilde g](x) \,dx + O(\eta^{d-2}).
\end{equation*}
Here $\tilde g \in L^2_0(\partial T)$ and $\|\tilde g\|_{L^2(\partial T)}$ is bounded by a universal constant.\label{pageref.22} For the integral of $\cD^\eta_T[\tilde g]$, we compute
\begin{equation*}
\begin{aligned}
\int_{\eta^{-1}\bT^d\setminus \ol T} \cD^\eta_T[\tilde g](x) dx &= \int_{\eta^{-1}\bT^d\setminus \ol T} \int_{\partial T} [-P^\eta_k(y-x)N^i_y + N^\ell_y \partial_{y_\ell} (G^\eta_k(y-x))^i](\tilde g(y))^i dy\\
&= \int_{\eta^{-1}\bT^d\setminus \ol T} \int_{\partial T} [\partial_{x_k} G^\eta_\Delta(y-x) N^i_y - N^\ell_y \partial_{x_\ell} (G^\eta_k(y-x))^i](\tilde g(y))^i dy\\
&=  \int_{\partial T} \int_{\partial T} [-G^\eta_\Delta(y-x)N^k_x N^i_y + N^\ell_y N^\ell_x (G^\eta_k(y-x))^i](\tilde g(y))^i dydx.
\end{aligned}
\end{equation*}
In view of the uniform (in $\eta$ and in $y$) integrability of $G^\eta_\Delta(y-x)$ and $G^\eta_k(y-x)$ over $x \in \partial T$, the integral above is uniformly bounded. 

For the integral of $\Gamma_k$, we only consider the case $d\ge 3$. By the estimate $|\Gamma_k(x)|\le C|x|^{-(d-2)}$ (for $d\ge 3$) and a direct computation, we have
\begin{equation*}
\left|\int_{\eta^{-1}\bT^d\setminus \ol T} \Gamma_k(x)\,dx \right| \le C\eta^{-2}.
\end{equation*}
We then conclude that \eqref{eq:chiavg} holds for $d\ge 3$.

\smallskip

\noindent\underline{Step 3: the proof for \eqref{eq:omeavg}.}\\
We need to use the restriction and extension operators in $\eta^{-1}\bT^d\setminus \ol T$. By rescaling the results in Proposition \ref{prop:restr}, we get a restriction operator $\cR$ for vector fields $H^1(\eta^{-1}\bT^d)$ that satisfies
\begin{equation*}
\|\nabla(\mathcal{R}u)\|_{L^2(\frac1\eta \bT^d \setminus \ol T)} 
\le C\left[\|\nabla u\|_{L^2(\frac{1}{\eta} \bT^d)} + \eta \kappa_\eta\|u\|_{L^2(\frac1\eta \bT^d)}\right],
\end{equation*}
where $C$ depends only on $d$ and $T$. This restriction operator then defines an extension operator $\mathcal{E}$ for mean-zero pressure fields in $L^2_0(\eta^{-1}\bT^d\setminus \ol T)$; for $p \in L^2_0(\eta^{-1}\bT^d\setminus \ol T)$, $\mathcal{E}p$ is determined by the relation
\begin{equation*}
\int_{\eta^{-1}\bT^d} (\mathcal{E}p) \,\nabla\cdot u = \int_{\eta^{-1}\bT^d \setminus \ol T} p\,\nabla\cdot\,(\mathcal{R} u), \qquad \forall u \in H^1(\eta^{-1}\bT^d).
\end{equation*}
Let $\langle \omega^\eta_k\rangle_{\bT^d_{\eta,f}}$ denote the mean value of $\omega^\eta_k$ in $\eta^{-1}\bT^d \setminus \ol T$. In view of the second inequality in \eqref{eq:chiavg} and that the volume fraction of the hole is of order $\eta^d$, we check that $|\langle \omega^\eta_k\rangle_{\bT^d_{\eta,f}}|$ is of order $O(\eta^d)$. By repeating the arguments that lead to \eqref{eq:punif}, we can find a universal constant $C$ such that
\begin{equation}
\label{eq:omeflbdd}
\|\mathcal{E} (\omega^\eta_k - \langle \omega^\eta_k\rangle_{\bT^d_{\eta,f}})\|_{L^2(\frac1\eta \bT^d)} \le \begin{cases} C \quad &\text{if } d\ge 3,\\
C|\log \eta|^{\frac12} \quad &\text{if } d=2.
\end{cases}
\end{equation}
It is also clear that $\mathcal{E}(\omega^\eta_k-\langle \omega^\eta_k\rangle_{\bT^d_{\eta,f}})$ and $\omega^\eta_k - \langle \omega^\eta_k\rangle$ only differ by a constant in $\eta^{-1}\bT^d\setminus \ol T$. Moreover, this constant is precisely
\begin{equation*}
\langle \omega^\eta_k\rangle - \langle \omega^\eta_k\rangle_{\bT^d_{\eta,f}}-\left|\eta^{-1}\bT^d\setminus \ol{T}\right|^{-1}  \int_T \mathcal{E} (\omega^\eta_k -\langle \omega^\eta_k\rangle),
\end{equation*}
which is of order $O(\eta^d)$. Hence, in \eqref{eq:omeflbdd} we can replace the extended function by $\omega^\eta_k-\langle \omega^\eta_k\rangle$. This proves \eqref{eq:omeavg}.
\end{proof}

\subsection{Asymptotic limits and quantitative estimates}

Thanks to the explicit formulas for the solutions $\chi^\eta_k$ and $\omega^\eta_k$ to the cell problem \eqref{eq:etacell}, we are able to characterize the limits of $(v^\eps_k,q^\eps_k)$, and, furthermore, to quantify some of the convergence results. We first prove Theorem \ref{lem:cellqual} stated in the Introduction.

\begin{proof}[Proof of Theorem \ref{lem:cellqual}]\ \\
\underline{Step 1: estimates for $\nabla v^\eps_k$.}\\
By periodicity we decompose the defining integral of $\|\nabla v^\eps_k\|^2_{L^2(\Omega)}$ into pieces over $\eps$-cubes that intersect $\Omega$. Using the definition \eqref{eq:vqeps} and by rescaling the estimate of $\|\nabla \chi^\eta_k\|^2_{L^2}$ in \eqref{eq:chiunif}, we see that $\|\nabla v^\eps_k\|^2_{L^2}$ in each $\eps$-cube is of order $O(\eps^{d-2} \kappa_\eta^2)$. The desired result then follows by counting the number of $\eps$-cubes in $\Omega$, which is of order $O(\eps^{-d})$. 

\smallskip

\noindent\underline{Step 2: weak convergence results for $\nabla v^\eps_k$ in the critical setting.}\\
The result is due to the periodicity of $v^\eps_k$ and can be viewed as a version of the Riemann-Lebesgue lemma. Because $\nabla v^\eps$ is uniformly bounded in $L^2(\Omega)$, it suffices to prove, for each $\ell \in \{1,\dots,d\}$,
\begin{equation*}
\int_{\R^d} \partial_\ell v^\eps_k(x) \varphi(x) dx \to 0, \qquad \forall \varphi \in C^\infty_c(\Omega;\R).
\end{equation*}
For each fixed $\eps$, let $\mathcal{I}_\eps$ be the set of indices $j \in \Z^d$ such that $Q_{\eps,j} = \eps(j+Q_1)$ is contained in $\Omega$, and let $\mathcal{J}_\eps$ be the set of $j$ such that $Q_{\eps,j} \cap \partial \Omega$ is non-empty. Hence, $\mathcal{I}_\eps$ and $\mathcal{J}_\eps$, take into account the interior cubes and the boundary $\eps$-cubes, respectively, that have non-empty intersection with $\Omega$. Let $d\ge 3$ for the moment. The above integral is then rewritten as
\begin{equation*}
\sum_{i\in \mathcal{I_\eps}} \int_{Q_{\eps,i}} \partial_\ell v^\eps_k \varphi +  \sum_{j\in \mathcal{J_\eps}} \int_{Q_{\eps,j} \cap \Omega} \partial_\ell v^\eps_k \varphi \,=\, (\eps\eta)^{d-1}\sum_{i\in \mathcal{I_\eps}} \int_{Q_{\frac1\eta,i}} (\partial_\ell \chi^\eta_k)(y) \varphi(\eps \eta y)\,dy 
+ \sum_{j\in \mathcal{J_\eps}} \int_{Q_{\eps,j} \cap \Omega} \partial_\ell v^\eps_k \varphi.
\end{equation*}
We used the change of variable $\frac{x}{\eps\eta} \mapsto y$ in the last step for the interior integrals. Because $\chi^\eta_k$ is periodic, each of the interior integrals for $i\in \mathcal{I}_\eps$ is computed as
\begin{equation*}
\int_{Q_{\frac1\eta,i}} (\partial_\ell \chi^\eta_k)(y) \varphi(\eps \eta y) = \int_{\frac1\eta Q_1} (\partial_\ell \chi^\eta_k)(y)\left[\varphi(\eps\eta y + \eta^{-1} i) - \varphi(\eta^{-1} i)\right].
\end{equation*}
Since $\varphi$ is smooth, by Taylor's theorem we also have
\begin{equation*}
\left|\varphi(\eps\eta y + \eta^{-1} i) - \varphi(\eta^{-1} i)\right| \le \|\nabla \varphi\|_{L^\infty} |\eps \eta y| \le \eps\|\nabla \varphi\|_{L^\infty}, \qquad \forall y \in \frac1\eta Q_1.
\end{equation*}
Using the estimate \eqref{eq:chiunif}, and the fact that $\mathcal{I}_\eps$ has cardinality of order $O(\eps^{-d})$, we see that the total contribution from interior integrals is bounded by
\begin{equation*}
\left|(\eps\eta)^{d-1} \sum_{i \in \mathcal{I}_\eps} \int_{Q_{\frac1\eta,i}} (\partial_\ell \chi^\eta_k)(y) \varphi(\eps \eta y) \right| \le \eta^{d-1}\|\nabla \varphi\|_{L^\infty} |Q_{\frac1\eta}|^{\frac12} \le \eta^{\frac{d-2}{2}}\|\nabla \varphi\|_{L^\infty},
\end{equation*}

For a typical boundary cube $Q_{\eps,j}$, with $j \in \mathcal{J}_\eps$, the integral over $Q_{\eps,j}\cap \Omega$ is controlled by
\begin{equation*}
\left|\int_{Q_{\eps,j}\cap \Omega} (\partial_\ell v^\eps_k) \varphi\right| \le \|\partial_\ell v^\eps_k\|_{L^2(Q_\eps)} \eps^{\frac d2} \|\varphi\|_{L^\infty} \le C\eps^{d-1}\eta^{\frac{d-2}{2}}.
\end{equation*}
The estimate above is uniform in $j$. Adding these estimates for $j$ in $\mathcal{J}_\eps$, which has cardinality of order $O(\eps^{-d+1}$), we see that the total contribution from the boundary integrals is of order $O(\eta^{\frac{d-2}{2}})$. We hence proved \eqref{eq:gradvconv} for $d\ge 3$. 

When $d=2$, there is a further factor of order $|\log\eta|^{-1}$ in the definition of $v^\eps_k$ in terms of $\chi^\eta_k$; see \eqref{eq:vqeps}. On the other hand, $\|\nabla \chi^\eta_k\|_{L^2(\eta^{-1}\bT^d)}$ is of order $|\log\eta|^\frac{1}{2}$. Taking those modifications into account, we verify that, given $\varphi \in C^1(\Omega)$, the integral of $(\partial_\ell v^\eps_k)\varphi$ is of order $O(|\log\eta|^{-\frac12})$ and it vanishes in the limit as $\eta$ and $\eps$ go to zero.

\smallskip

\noindent\underline{Step 3: The limit of $v^\eps_k$.}\\
We first consider the setting of $d\ge 3$ and take $p = \frac{2d}{d-2}$. We compute
\begin{equation*}
\|v^\eps_k-M^{-1}e_k\|_{L^p(\Omega)}^p  = \sum_{j\in \mathcal{I}_\eps} \int_{Q_{\eps,j}\cap \Omega} |v^\eps_k - M^{-1}e_k|^p \le \sum_{j\in \mathcal{I}_\eps} \int_{Q_{\eps,j}} |v^\eps_k - M^{-1}e_k|^p,
\end{equation*}
where $\mathcal{I}_\eps$ denotes the set of indices $j$ for which $Q_{\eps,j}$ has nonempty intersection with $\Omega$. By periodicity, the integral over $Q_{\eps,j}$ is uniform in $j$ and it is computed and controlled by 
\begin{equation}\label{e.veps-M1}
\int_{Q_\eps} |v^\eps_k-M^{-1} e_k|^p \le C\int_{Q_\eps} |v^\eps_k - \langle v^\eps_k\rangle|^p + |\langle v^\eps_k\rangle - M^{-1}e_k|^p.
\end{equation}
Here, $Q_\eps$ is the $\eps$-cube centered at $0$ and $\langle v^\eps_k\rangle$ is the average of $v^\eps_k$ in $Q_\eps$. By scaling, we check that $\langle v^\eps_k\rangle = \langle \chi^\eps_k\rangle_{\eta^{-1}\bT^d}$. The inequality \eqref{eq:chiavg} then shows
\begin{equation}\label{e.veps-M2}
\int_{Q_{\eps}} |\langle v^\eps_k\rangle - M^{-1}e_k|^p \le C\eps^d \eta^{2d}.
\end{equation}
On the other hand, we have the Sobolev embedding
\begin{equation}\label{e.veps-M3}
\int_{Q_\eps} |v^\eps_k - \langle v^\eps_k\rangle|^{p} \le C\|\nabla v^\eps_k\|_{L^2(Q_\eps)}^p = C\left[(\eps\eta)^{\frac{d-2}{2}}\|\nabla \chi^\eta_k\|_{L^2(\eta^{-1}\bT^d)}\right]^p.
\end{equation}
Here, the bounding constant $C$ is independent of $\eps$, because the Sobolev embedding with critical component $p = 2d/(d-2)$ is scaling-invariant. The integral above is hence of order $O(\eps^d\eta^d)$. Since $\mathcal{I}_\eps$ has cardinality of order $\eps^{-d}$, we combine the estimates above and check that \eqref{eq:vquant} holds for $d\ge 3$.

\smallskip

Finally we consider the setting of $d=2$. In view of the definition \eqref{eq:vqeps} and the formula \eqref{eq:etachi}, we compute and get, for $x \in Q_{\eps,j}$ but outside $\eps(j+\eta \ol T)$,
\begin{equation*}
v^\eps_k(x) - M^{-1}e_k = v^\eps_k(x) - \frac{1}{4\pi} e_k = \frac{1}{4\pi} \frac{\log|x/\eps|}{|\log \eta|} e_k + \frac{1}{|\log \eta|} \cD^\eta_T[\tilde g](\frac{x}{\eps\eta}) + O(|\log \eta|^{-1}).
\end{equation*}
Here, $\tilde g = g - \langle g\rangle_{\partial T}$ is specified in Lemma \ref{lem:chiomega}. For $x$ in the holes, $v^\eps_k$ is zero and the above is $-M^{-1}e_k$. Let $p = 2$; we have
\begin{equation}
\label{eq:vlim-1}
\|v^\eps_k - M^{-1}e_k\|^2_{L^2(\Omega)} \le 2\sum_{j\in \mathcal{I}_\eps} \int_{Q_{\eps,j}\cap \Omega} |v^\eps_k - \langle v^\eps_k\rangle|^2 + |\langle v^\eps_k\rangle - M^{-1}e_k|^2.
\end{equation}
Here, $\langle v^\eps_k\rangle$ is the average on the $\eps$-cube $Q_\eps$. The representation formula for $v^\eps_k - M^{-1}e_k$ together with rescaling show that
\begin{equation*}
\left|\langle v^\eps_k\rangle - M^{-1}e_k\right| \le C|\log \eta|^{-1}\left( 1 + \left|\langle \log |x|\rangle_{Q_1}\right| + \left|\langle \cD^\eta_T[\tilde g]\rangle_{\eta^{-1}\bT^d\setminus \ol T}\right|\right).
\end{equation*}
From the computations in the proof of Corollary \ref{coro:chi}, the right hand side above is of order $O(|\log \eta|^{-1})$. On the other hand, the usual Poincar\'e inequality shows
\begin{equation*}
\|v^\eps_k - \langle v^\eps_k\rangle \|^2_{L^2(Q_{\eps,j})} \le C\eps^2 \|\nabla v^\eps_k\|^2_{L^2(Q_{\eps,j})} = C\eps^2 |\log \eta|^{-2}\|\nabla \chi^\eta_k\|_{L^2(\eta^{-1}\bT^d)} \le C\eps^2|\log \eta|^{-1}.
\end{equation*}
We combine those results and use them in \eqref{eq:vlim-1}. Since the cardinality of $\mathcal{I}_\eps$ is of order $O(\eps^{-2})$, we check that \eqref{eq:vquant} holds also for $d=2$.
\end{proof}


For the pressure field $q^\eps_k$ defined in \eqref{eq:vqeps}, we have the following results.

\begin{lemma}
\label{lem:vq}
Under the assumptions of Theorem \ref{lem:cellqual}, there exists a universal constant $C > 0$ such that the following holds for each $k=1,\ldots\, d$.
\begin{itemize}
\item[(1)] For all $d\ge 2$, we have
\begin{equation}
\label{eq:qquant}
\|q^\eps_k\|_{L^{2}(\Omega)} \le C\sigma_\eps^{-1}.
\end{equation}

\item[(2)] We have, for $d\ge 3$,
\begin{equation}
\label{eq:qquant1}
\left|\int_\Omega q^\eps_k \phi\right| \le C\left( \eta^{\frac{d-2}{2}}\|\nabla \phi\|_{L^2(\Omega)} + \sigma_\eps^{-1}\eta^{\frac d2}\|\phi\|_{L^2(\Omega)}\right).
\end{equation}
For $d=2$, the inequality still holds with the first $\eta^{\frac{d-2}{2}}$ replaced by $|\log \eta|^{-\frac12}$.
\end{itemize}
\end{lemma}

\begin{proof} By the definition in \eqref{eq:vqeps}, $q^\eps_k$ is essentially a rescaling of $\omega^\eta_k$ in each $\eps$-cube $Q_{\eps,j} = \eps(j+Q_1\setminus \ol \eta T)$. Let $\mathcal{I}_\eps$ denote the indices $j$'s such that $Q_{\eps,j}$ has non-empty intersection with $\Omega$. To estimate $q^\eps_k$ in $L^2(\Omega)$, we break it into pieces on the cubes $\{Q_{\eps,j}\}$ for $j\in \mathcal{I}_\eps$, and only need to study $q^\eps_k$ on a typical $\eps$-cube, say $Q_\eps = Q_{\eps,0}$ that is centered at the origin. Let $\langle q^\eps_k\rangle$ be the average of $q^\eps_k$ in $Q_\eps$, then clearly 
\begin{equation*}
\langle q^\eps_k\rangle = \begin{cases} (\eps\eta)^{-1}\langle \omega^\eta_k\rangle_{\eta^{-1}\bT^d} \quad &\text{if } d\ge 3,\\
(\eps\eta|\log\eta|)^{-1}\langle \omega^\eta_k\rangle_{\eta^{-1}\bT^d} \quad &\text{if } d=2.
\end{cases}
\end{equation*}
By rescaling the results in \eqref{eq:omeavg} and \eqref{eq:chiavg}, we obtain
\begin{equation*}
\|q^\eps_k - \langle q^\eps_k\rangle \|_{L^2(Q_{\eps})} \le \begin{cases} C(\eps\eta)^{\frac{d-2}{2}} \quad &\text{if } d \ge 3,\\
C|\log \eta|^{-\frac12} \quad &\text{if } d=2,
\end{cases}
 \quad \text{and} \quad |\langle q^\eps_k\rangle| \le \begin{cases} C\eps^{-1}\eta^{d-1} \quad &\text{if } d\ge 3,\\
 C\eps^{-1}\eta|\log\eta|^{-1} \quad &\text{if } d=2.
 \end{cases}
\end{equation*}
In view of the second part of the inequalities above, we check that the first inequality still holds if $\langle q^\eps_k\rangle$ is removed. The estimate in \eqref{eq:qquant} then follows from the definition of $\sigma_\eps$ and from the fact that the cardinality of $\mathcal{I}_\eps$ is of order $O(\eps^{-d})$. 

\smallskip

To prove \eqref{eq:qquant1}, fix any $\phi \in H^1_0(\Omega^\eps)$. Extend $\phi$ by zero outside $\Omega$, so $\phi$ is defined in all $Q_{\eps,j}$ for $j \in \mathcal{I}_\eps$. We then compute
\begin{equation*}
\int_\Omega q^\eps_k \phi = \sum_{j\in \mathcal{I}_\eps} \int_{Q_{\eps,j}} q^\eps_k(x) \phi(x) dx.
\end{equation*}
For the integral over the cube $Q_{\eps,j}$, let $\langle \phi\rangle_j$ denote the mean value of $\phi$ over $Q_{\eps,j}$, then
\begin{equation*}
  \int_{Q_{\eps,j}} q^\eps_k(x) \phi(x) dx =  \int_{Q_{\eps,j}} (q^\eps_k(x)-\langle q^\eps_k\rangle) (\phi(x) -\langle \phi\rangle_j) dx + \left(\int_{Q_{\eps,j}}\phi(x) dx\right) \fint_{Q_{\eps,j}} q^\eps_k
\end{equation*} 
In view of the earlier estimates obtained for $q^\eps_k$, and using the Poincar\'e-Wirtinger inequality
\begin{equation*}
\|\phi - \langle \phi\rangle_j\|_{L^2(Q_{\eps,j})} \le C\eps \|\nabla \phi\|_{L^2(Q_{\eps,j})},
\end{equation*}
which holds uniformly for all $j$, we obtain, for $d\ge 3$,
\begin{equation*}
\begin{aligned}
\left|\int_\Omega q^\eps_k \phi\right| &\le \sum_{j\in \mathcal{I}_\eps} C(\eps\eta)^{\frac{d-2}{2}} \eps \|\nabla\phi\|_{L^2(Q_{\eps,j})} + C\eps^{-1}\eta^{d-1}\sum_{j\in\mathcal{I}_\eps}\left|\int_{Q_{\eps,j}}\phi\right|\\
&\le C\eta^{\frac{d-2}{2}} \|\nabla \phi\|_{L^2(\Omega)} + C\eps^{-1}\eta^{d-1}\|\phi\|_{L^1(\Omega)}\\
&\le C(\eta^{\frac{d-2}{2}}\|\nabla \phi\|_{L^2(\Omega)} + \sigma^{-1}_\eps \eta^{\frac d2}\|\phi\|_{L^2(\Omega)}).
\end{aligned}
\end{equation*}
For $d=2$, after some modifications in the intermediate steps, we get
\begin{equation*}
\left|\int_{\Omega} q^\eps_k \phi\right| \le C\left(|\log \eta|^{-\frac12}\|\nabla \phi\|_{L^2} + \sigma^{-1}_\eps \eta \|\phi\|_{L^2} \right).
\end{equation*}
The proof of the lemma is now complete.
\end{proof}

\begin{lemma}
  \label{lem:stressbound}
Under the assumptions of Theorem \ref{lem:cellqual}, we have
  \begin{equation}
	\left|\nabla v^\eps_k(x)\right| + |q^\eps_k(x)| \le C\eps^{-1}\kappa_\eta^2, \qquad x\in \eps \partial Q_1.
	\label{eq:vqstressbound}
  \end{equation}
\end{lemma}
\begin{proof} The result follows from the explicit formulas of $v^\eps_k$ and $p^\eps_k$. We only provide the details for $d\ge 3$; the case for $d=2$ is similar. By rescaling, by the formula \eqref{eq:chiform} of $(\chi^\eta_k,\omega^\eta_k)$ and the perturbative formula \eqref{eq:Gperturb}, we have
  \begin{equation*}
	\begin{aligned}
	  &\nabla v^\eps_k(x) = \frac{1}{\eps \eta} (\nabla \chi^\eta_k)(y) = \frac{1}{\eps\eta}\left[\nabla\Gamma_k(y) + \eta^{d-1}(\nabla R_k)(\eta y) + (\nabla \cD^\eta_T[\tilde g])(y)\right], \\ 
	  &q^\eps_k(x) = \frac{1}{\eps\eta} \omega^\eta_k(y) = \frac{1}{\eps\eta} \left[\theta_k(y) + \eta^{d-1}h_k(\eta y) + \cP^\eta_T[\tilde g](y)\right].
  \end{aligned}
  \end{equation*}
  where $y = \frac{x}{\eps\eta}$ is the rescaled variable and belongs to $\eta^{-1}\bT^d\setminus \ol T$. We are interested in the case of $x\in \eps\partial Q_1$, i.e. $y\in \eta^{-1}\partial Q_1$. Such a point $y$ is far away from the singularities of the fundamental solutions $(\Gamma_k,\theta_k)$ and from the singularities of the integral operators $\cD^\eta_T$ and $\cP^\eta_T$. Hence, all terms on the right hand sides can be computed rather easily, and the desired estimate follows from explicit computations and from the observation that $|y| = O(\eta^{-1})$, $d(y,\partial T) = O(\eta^{-1})$ and from the estimate \eqref{eq:gbdd} for $\tilde g$. 
\end{proof}


\section{Homogenization errors in the dilute Darcy's regime}
\label{sec:app:sup}

Our goal here is to prove Theorem \ref{thm:homerr}. We establish the convergence rate for the homo\-genization of \eqref{eq:hetstokes} in the dilute super-critical case, namely \eqref{eq:err_supc}, in the general setting, i.e. without assuming the homogenized solution satisfies $u\rvert_{\partial \Omega} = 0$.

\subsection{Equations for the cell functions}

In the construction of $\Omega^\eps$, we removed (filled in) some holes of $\eps \R^d_{\rm f}$ near the boundary $\partial \Omega$; see \eqref{eq:Oepsdef}. As a consequence, there is a mismatch between the sets $\Omega^\eps$ and $\Omega\cap (\eps \R^d_{\rm f})$. To derive convergence rates for homogenization, we need to use the equations for the cell functions $(v^\eps_k,q^\eps_k)$, $k=1,\dots,d$, in the modified domain $\Omega^\eps$. Direct computations using integration by parts show that the first line of \eqref{eq:rscell} should be changed to (with a slight abuse, we use the same notation for the solutions of \eqref{eq:vqbdhole} and \eqref{eq:rscell})
\begin{equation}
  \label{eq:vqbdhole}
  -\Delta v^\eps_k + \nabla q^\eps_k = \frac{e_k}{\sigma_\eps^2} + s_{\eps,k} \qquad \text{in } \; \Omega^\eps,
\end{equation}
where $s_{\eps,k} \in H^{-1}(\Omega^\eps)$ is defined by
\begin{equation}
  \langle s_{\eps,k},\varphi\rangle_{H^{-1},H^1_0} := -\sum_{z\in \mathcal{I}'_\eps} \left[\int_{\Omega \cap \eps(z+\eta T)} \frac{e_k}{\sigma_\eps^2} \cdot \varphi + \int_{\ol \Omega \cap \eps(z + \partial (\eta \ol T))} \frac{\partial (v^\eps_k,p^\eps_k)}{\partial \nu} \cdot \varphi\right], \quad \varphi \in H^1_0(\Omega^\eps). 
	\label{eq:sepsk}
\end{equation}
Here, the sum over $z$ is taken over $\mathcal{I}'_\eps:=\{z \in \Z^d \,:\, \eps(z+Q_1) \cap \partial \Omega \ne \varnothing\}$. In other words, they correspond to the removed (filled) holes near $\partial \Omega$. 
Compare the equations \eqref{eq:vqbdhole} in $\Omega^\eps$ and \eqref{eq:rscell} in $\eps\R^d_{\rm f}$, we see the latter have an extra term $s_{\eps,k}$ which accounts for the mismatch between the two domains. We will see later that this term causes some difficulty in the convergence rates analysis, more so in the dilute setting than in the classical setting \cite{MR4432951}. 

For each $t>0$ small, let $\Omega_{(t)} := \{x\in \Omega \,:\, d(x,\partial \Omega) \le t\}$ defines a $t$-neighborhood of $\partial \Omega$. Recall that $K_\eps$ is the  neighborhood of $\partial \Omega$ in which holes were filled.

\begin{lemma}
  \label{lem:sepsk-2}
Assume that $\varphi \in H^1_0(\Omega)$ satisfies $\varphi = 0$ in $\Omega\setminus \Omega^\eps$. Then there exists a universal constant $C>0$ such that  
\begin{equation}
\left|\langle s_{\eps,k},\varphi\rangle\right| \le C\sqrt{\eps}\sigma_\eps^{-1}\|\nabla \varphi\|_{L^2(\Omega_{(2\ep)})}.
  \label{eq:sepskbdd-2}
\end{equation}
\end{lemma}
\begin{proof}
  Extend $\varphi$ by zero in $\R^d\setminus \Omega$ so $\varphi \in H^1_0(\R^d)$. Multiply the equation \eqref{eq:rscell} and take the integral over the domain $\eps\R^d_{\rm f}\setminus(\Omega\setminus K_\eps)$. Via integration by parts we get
\begin{equation*}
  -\int_{\partial K_\eps \cap \Omega} \frac{\partial(v^\eps_k,q^\eps_k)}{\partial \nu} \cdot \varphi + \sum_{z\in \mathcal{I}'_\eps} \int_{\eps(z+\eta\partial T)} \frac{\partial(v^\eps_k,q^\eps_k)}{\partial \nu} \cdot \varphi = \int_{K_\eps\cap \eps\R^d_{\rm f}} \sigma_\eps^{-2}e_k\cdot \varphi - (\nabla v^\eps_k - q^\eps_k \mathbf{I}):\nabla \varphi\,dx.   
\end{equation*}
Plugging this into the definition \eqref{eq:sepsk}, we obtain
\begin{align}
\begin{split}
  \langle s_{\eps,k},\varphi\rangle =\ & -\sigma^{-2}_\eps \int_{K_\eps} e_k\cdot \varphi + \int_{K_\eps} (\nabla v^\eps_k -q^\eps_k \mathbf{I}):\nabla \varphi + \int_{\partial K_\eps \cap \Omega} N_y\cdot (\nabla v^\eps_k - q^\eps \mathbf{I})\cdot \varphi\,dy\\
   =\ &I_1+I_2+I_3.
   \end{split}
  \label{eq:sepsk-2}
\end{align}

For $I_1$, observe that $K_\eps$ is a subset of $\Omega_{(2\eps)} := \{x\in \Omega\,:\, d(x,\partial \Omega) \le 2\eps\}$. Using Poincar\'e's inequality $\|\varphi\|_{L^2(\Omega_{(2\eps)})}\le C\eps\|\nabla\varphi\|_{L^2(\Omega_{(2\eps)})}$, we have 
\begin{equation*}
  |I_1| \le \sigma_\eps^{-2}\|e_k\|_{L^2(K_\eps)} \|\varphi\|_{L^2(K_\eps)} \le C\eps^\frac32\sigma_\eps^{-2} \|\nabla \varphi\|_{L^2(\Omega_{(2\eps)})}.
\end{equation*}

For $I_2$, we break the integral over $K_\eps$ to the sum of integrals over $\eps$-cubes $\{\eps(z+Q_1)\}$, $z\in \mathcal{I}'_\eps$. Note that the square of the $L^2$ norms of $\nabla v^\eps_k$ and $q^\eps_k$ in each $\eps$-cube is of order $O(\eps^{d-2}\kappa_\eta^2)$ (see Theorem \ref{lem:cellqual} and Lemma \ref{lem:vq}), and that the number of $\eps$-cubes included in $K_\eps$ is of order $O(\eps^{-d+1})$. We deduce 
\begin{equation*}
  |I_2| \le \|\nabla v^\eps_k -q^\eps_k\|_{L^2(K_\eps)} \|\nabla \varphi\|_{L^2(K_\eps)} \le C(\eps^{-1}\kappa_\eta^2)^{\frac12}  \|\nabla \varphi\|_{L^2(K_\eps)} \leq C\eps^{\frac12}\sigma_\eps^{-1}\|\nabla \varphi\|_{L^2(\Omega_{(2\ep)})}.
\end{equation*}

For $I_3$, we use the $L^\infty$ estimate \eqref{eq:vqstressbound} of $\nabla v^\eps_k - q^\eps_k\mathbf{I}$ on $\eps\partial (z+Q_1)$, $z\in \Z^d$, and conclude that
\begin{equation*}
|I_3| \le C\sigma_\eps^{-1}\kappa_\eta^2\int_{\partial K_\eps \cap \Omega} |\varphi(x)|\,dx 
\end{equation*}
Note that $\partial K_\eps \cap \Omega$ consists of sides of $\eps$-cubes $\eps(z+Q_1)$ with $z\in \mathcal{I}'_\eps$. In view of the elementary estimate
\begin{equation*}
  \int_{\eps(z+\partial Q_1)} |\varphi|^2 \le C\eps\int_{\eps(z+Q_1)} |\nabla\varphi|^2 + C\eps^{-1}\int_{\eps(z+Q_1)} |\varphi|^2.
\end{equation*}
We obtain that 
\begin{equation*}
  |I_3| \le C\sigma_\eps^{-1}\kappa_\eta^2 \left(\sqrt{\eps}\|\nabla \varphi\|_{L^2(K_\eps)} + \eps^{-\frac12} \|\varphi\|_{L^2(K_\eps)} \right) \le C\sqrt{\eps}\sigma_\eps^{-1}\kappa_\eta^2 \|\nabla \varphi\|_{L^2(\Omega_{(2\ep)})}.
\end{equation*}
Note that $\eps^\frac32\sigma^{-2}_\eps = \sqrt{\ep}\kappa_\eta \sigma^{-1}_\eps$, and so the bound for $I_2$ dominates the others. We conclude that \eqref{eq:sepskbdd-2} holds.
\end{proof} 


\subsection{Discrepancy functions}

As usual, the quantification of convergence rates amounts to choosing correctors $r_\eps$ and $t_\eps$ for the velocity field and, respectively, the pressure field, which are small and which make the corrected discrepancy functions 
$u^\eps/(\sigma^2_\eps\wedge 1) - u - r_\eps$ and $p^\eps-p - t_\eps$ small. Our method is based on the classical two-scale expansion method, and this expansion suggests natural candidates for correctors, at least in the classical setting. Again, by formally treating $\eta$ as a parameter, the two-scale expansion inspires us to consider the following discrepancy functions:
\begin{equation}
  \label{eq:zetatau-super}
  \zeta^\eps := \frac{u^\eps(x)}{\sigma^2_\eps} - v^\eps_k(x)[f^k - \partial_k p](x),  \qquad 
  \tau^\eps :=p^\eps(x) - p(x) - \sigma^2_\eps q^\eps_k(x)[f^k-\partial_k p](x).
\end{equation}
Here, $(v^\eps_k,q^\eps_k)$ are the rescaled cell functions solution to \eqref{eq:rscell}. 
It turns out that, the above choice works perfectly in the (dilute) super-critical setting. This is expected because, in the classical setting where $\sigma_\eps =1$, the terms subtracted are precisely the leading ones in the two-scale expansion. In the critical and sub-critical settings, we will slightly modify the above to come up with proper discrepancy functions, see Appendix \ref{sec.app.quant}.

We can rewrite $f-\nabla p$ as $Mu$ where $u$ is the homogenized solution; see \eqref{eq:hpdesupc}. We consider the discrepancy functions $\zeta^\eps$ and $\tau^\eps$ defined in \eqref{eq:zetatau-super}, hence choosing the correctors
\begin{equation}
  \label{eq:corrector_supc}
  r_\eps = (v^\eps_k-M^{-1}e_k) (Mu)^k,\qquad
  t_\eps = \sigma^2_\eps q^\eps_k(Mu)^k.
\end{equation}
The smallness of $\|r_\eps\|_{L^2}$ and $\|t_\eps\|_{L^2}$ is immediately seen from \eqref{eq:vquant} and \eqref{eq:qquant}. 
To estimate the discrepancy functions $(\zeta^\eps,\tau^\eps)$, we check by direct computation (using \eqref{eq:hetstokes} and \eqref{eq:vqbdhole}) that they satisfy the following equations in the perforated domain $\Omega^\eps$:
\begin{equation}
  \label{eq:zteq_supc}
\left\{
\begin{aligned}
  &-\sigma_\eps^2 \Delta \zeta^\eps + \nabla \tau^\eps = \sigma_\eps^2 \nabla \cdot [v^\eps_k\nabla (Mu)^k] + \sigma_\eps^2 (\nabla v^\eps_k) \cdot \nabla (Mu)^k-\sigma_\eps^2 q^\eps_k \nabla (Mu)^k + \sigma^2_\eps \check{s},\\
&\nabla \cdot \zeta^\eps = -(v^\eps_k - M^{-1}e_k)\cdot \nabla (Mu)^k.
\end{aligned}
\right.
\end{equation}
Here, $\check{s} \in H^{-1}(\Omega)$ is defined by
\begin{equation}
  \label{eq:checksdef}
  \langle \check{s},\psi\rangle_{H^{-1},H^1_0} := \sum_{k=1}^d \langle -s_{\eps,k},(Mu)^k\psi\rangle_{H^{-1},H^1_0} \qquad \text{for} \; \psi \in H^1_0(\Omega).
\end{equation}

The equations above form a standard Dirichlet boundary value problem once augmented with boundary data at $\partial \Omega^\eps$. Since $u^\eps$ and the $v^\eps_k$'s all vanish at the boundaries of the holes, we get
\begin{equation}
  \label{eq:zetabdv}
  \zeta^\eps(x)  = -v^\eps_k(x)(Mu)^k(x)\mathbf{1}_{\partial \Omega}(x), \qquad x \in \partial \Omega^\eps.
\end{equation}
In the super-critical setting, the homogenized solution $u$ in \eqref{eq:hpdesupc} does not need to vanish on  $\partial \Omega$, and, hence, $\zeta^\eps \not\in H^1_0(\Omega^\eps)$ in general. This fact makes the quantification of convergence rates quite a difficult task for the super-critical setting, as is already the case in the classical setting \cite{MR4432951}.

\subsection{Proof of Theorem \ref{thm:homerr}: the cut-off function approach}
\label{subsec.dilute.quant}

In Appendix \ref{sec.app.qual} we give a unified treatment of the qualitative convergence. In Appendix \ref{sec.app.quant} we give a unified treatment of the qualitative convergence, further imposing that the homogenized solution $u$ satisfies $u=0$ on $\partial\Omega$ in the super-critical setting. We focus here on the most difficult case, namely the dilute super-critical case without assuming $u\rvert_{\partial \Omega} = 0$.

The system \eqref{eq:zteq_supc} is a special case of the non-homogeneous Stokes system \eqref{eq:vpsystem} with $\sigma = \sigma_\eps$, and 
\begin{equation}
  \label{eq:data_supc}
  \begin{aligned}
	&b = \sigma^2_\eps (\nabla v^\eps_k - q^\eps_k \mathbf{I})\cdot \nabla (Mu)^k, \quad && F = \sigma_\eps v^\eps_k \nabla (Mu)^k,\\
	&g = -(v^\eps_k-M^{-1}e_k) \cdot \nabla (Mu)^k, \quad && s = \sigma_\eps \check{s}, \qquad h = -v^\eps_k (Mu)^k.
  \end{aligned}
\end{equation}
Note that $g$ and $h$ above do satisfy the compatibility condition \eqref{eq:ghcompatible}. We can apply the basic energy estimate of Stokes system (see Proposition \ref{prop:zteq_supc}) to the pair $(\zeta^\eps,\tau^\eps)$. 

To this end, we use the so-called radial cut-off function of Wang, Xu and Zhang \cite{WXZ22}, to construct a proper vector field $H \in H^1(\Omega)$ to plug into \eqref{eq:benergy}. For $t>0$ sufficiently small, the authors of \cite{WXZ22} constructed a cut-off function $\phi_t$ satisfying $\phi_t = 0$ in $\Omega\setminus \Omega_{(3t)}$ and $\phi_t=1$ in $\Omega_{(t)}$ and the support of $\nabla \phi_t$ is contained in $\Omega_{(3t)}\setminus \Omega_{(2t)}$ and $|\nabla \phi_t| \le Ct^{-1}$. Moreover, in the support of $\nabla \phi_t$, $\nabla \phi_t(x)$ is proportional to $N_{x'}$, where $x' = Px$ is the orthogonal projection of $x$ to the boundary $\partial \Omega$.

For $t>0$ small to be chosen, let
\begin{equation}
  H(x) = \phi_t(x)v^\eps_k(x)(Mu)^k(x), \qquad x\in \Omega.
  \label{eq:altHdef}
\end{equation}
Then $H$ satisfies $H+\zeta^\eps \in H^1_0(\Omega)$ (due to $\phi_t = 1$ on $\partial \Omega$) and $H+\zeta^\eps = 0$ in the holes of $\Omega^\eps$ (due to the presence of $v^\eps_k$ in both functions).
By computation we check
\begin{equation*}
  \begin{aligned}
	\nabla H &= (\nabla \phi_t)v^\eps_k(Mu)^k + \phi_t(\nabla v^\eps_k)(Mu)^k + \phi_t v^\eps_k \nabla(Mu)^k,\\
	\nabla\cdot H &= (v^\eps_k-M^{-1}e_k)\cdot (\nabla \phi_t) (Mu)^k + \phi_t(v^\eps_k-M^{-1}e_k)\cdot \nabla (Mu)^k + (\nabla \phi_t)\cdot u. 
  \end{aligned}
\end{equation*}

We first estimate $\|\nabla H\|_{L^2(\Omega)}$. Since $\phi_t$ and $\nabla \phi_t$ are supported in $\Omega_{(3t)}$ and $\Omega_{(3t)}\setminus\Omega_{(2t)}$, respectively, we use H\"older's inequality and the bound $|\nabla \phi_t| \le Ct^{-1}$ and deduce
\begin{equation*}
  \begin{aligned}
  \|\nabla H\|_{L^2(\Omega)} &\le \sum_{k=1}^d C\left[t^{-1}\|v^\eps_k\|_{L^2(\Omega_{(3t)})}  + \|\nabla v^\eps_k\|_{L^2(\Omega_{(3t)})} \right](
  \|(Mu)^k\|_{L^\infty(\Omega_{(3t)})} + \|\nabla(Mu)^k\|_{L^\infty(\Omega_{(3t)})}) \\
  &\le \sum_{k=1}^d C\left[t^{-1}\|v^\eps_k\|_{L^2(\Omega_{(3t)})}  + \|\nabla v^\eps_k\|_{L^2(\Omega_{(3t)})} \right]\|(Mu)^k\|_{W^{1,\infty}(\Omega)}.
  \end{aligned}
\end{equation*}
For the $L^2(\Omega_{(3t)})$ estimates of $v^\eps_k$ and of $\nabla v^\eps_k$, we break the integration region $\Omega_{(3t)}$ into $\eps$-cubes. The number of such cubes is of order $O(t\eps^{-d})$. Inside each $\eps$-cube, the square of the $L^2$ norm of $\nabla v^\eps_k$ is of order $O(\eps^{d-2}\kappa_\eta^2)$, and that of $v^\eps_k$ is of order $O(\eps^d)$. We conclude that
\begin{equation*}
  \begin{aligned}
	\sigma_\eps\|\nabla H\|_{L^2(\Omega)} &\le C\sigma_\eps \left[t^{-1}(\eps^d t\eps^{-d})^{\frac12} + (\eps^{d-2}\kappa_\eta^2 t\eps^{-d})^{\frac12} \right]\|u\|_{W^{1,\infty}(\Omega)} \\
	&=C(\sigma_\eps t^{-\frac12} + t^{\frac12})\|u\|_{W^{1,\infty}(\Omega)}.
\end{aligned}
\end{equation*}

Now we estimate $\|\nabla\cdot H\|_{L^2(\Omega)}$. By the same reasoning as above, and choosing $p=2d/(d-2), q=d$ for $d\ge 3$, and $p=2, q=\infty$ for $d=2$, we get
\begin{equation}
\label{eq:divHbdd}
\begin{aligned}
\|\nabla\cdot H\|_{L^2} \le &\sum_{k=1}^d C\|v^\eps_k-M^{-1}e_k\|_{L^p(\Omega_{(3t)})} \|1\|_{L^q(\Omega_{(3t)})} \left[ t^{-1}  \|(Mu)^k\|_{L^2(\Omega_{(3t)})} + \|\nabla(Mu)^k\|_{L^2(\Omega_{(3t)})}\right]\\
&\qquad + \|\nabla \phi_t \cdot u\|_{L^2(\Omega_{(3t)})}\\
\le &C t^{\frac1q}\big(t^{-\frac12}\|u\|_{H^1} + t^{\frac12}\|u\|_{H^2}\big) \max_k \|v^\eps_k-M^{-1}e_k\|_{L^p(\Omega_{(3t)})} + \|\nabla \phi_t \cdot u\|_{L^2(\Omega_{(3t)})}.
\end{aligned}
\end{equation}
By \eqref{e.veps-M1}-\eqref{e.veps-M3}, in each $\eps$-cube, $\|v^\eps_k - M^{-1}e_k\|^p_{L^p}$ is of order $O(\eps^d \eta^d)$ for $d\ge 3$ and is of order $O(\eps^2 \kappa_\eta^2)$ for $d=2$. We conclude that
\begin{equation*}
\max_k \|v^\eps_k - M^{-1}e_k\|_{L^p(\Omega_{(3t)})} \le Ct^{\frac{1}{p}} \kappa_\eta, \qquad d\ge 2.
\end{equation*}
The first part in the last line of \eqref{eq:divHbdd} is bounded by $C(\kappa_\eta\|u\|_{H^1} + t\kappa_\eta\|u\|_{H^2})$.

Finally, we control $\|\nabla \phi_t\cdot u\|_{L^2}$. It is for this term that we use the special property of the radial cut-off function: the alignement of $\nabla \phi_t$ with $N_{x'}$, see above. This property allows us to use the condition $N \cdot u=0$ on $\partial \Omega$. 

We assume that $f\in C^{1,\frac12}(\ol \Omega)$ so that the $C^2$ regularity (up to the boundary) for the second order elliptic equation with constant coefficient for $p$ holds. More precisely,
\begin{equation*}
\|\nabla p\|_{L^\infty(\Omega)} + \|\nabla^2 p\|_{L^\infty(\Omega)} \le C\|f\|_{C^{1,\frac12}(\ol \Omega)}.
\end{equation*}
Since $Mu = f-\nabla p$, we can replace the $\nabla p$ on the left by $Mu$ and replace $\nabla^2 p$ by $\nabla u$, and the estimate still holds. Now, for any $x\in \mathrm{supp}(\nabla \phi_t) \subset \Omega_{(2t)}$, 
\begin{equation*}
  \nabla \phi_t \cdot u(x) = |\nabla \phi_t (x)|N_{Px}\cdot u(x) \le Ct^{-1}|N_{Px}\cdot u(x)|.
\end{equation*}
Note that $N_{Px}\cdot u(Px)=0$, we obtain
\begin{equation*}
  \left|N_{Px}\cdot u(x) \right|= \left| N_{Px}\cdot (u(x)-u(Px))\right| \le \|\nabla u\|_{L^\infty} t. 
\end{equation*}
This estimate holds uniformly in $x$ and, hence, we conclude that $\|\nabla \phi_t \cdot u\|_{L^2(\Omega)}\le C\sqrt{t}$. 

Combining all the estimates above, and choosing $t = \sigma_\eps$, we see that
\begin{equation*}
\sigma_\eps \|\nabla H\|_{L^2(\Omega)} + \|\nabla\cdot H\|_{L^2(\Omega)} \le C(\sqrt{\sigma_\eps} + \kappa_\eta) \|u\|_{W^{1,\infty}(\Omega)}  + C\eps\|u\|_{H^2(\Omega)}.
\end{equation*}
Note that, under the assumption that $f\in C^{1,\frac12}(\ol \Omega)$, the boundedness of $\|u\|_{H^2} + \|\nabla u\|_{L^\infty}$ is guaranteed. We conclude by applying the energy estimate of Proposition \ref{prop:zteq_supc}.

\begin{remark} 
If one assumes that the solution vanishes on the outer boundary $\partial\Omega$, the proof is much shorter, see \cite{Allaire91-2}. This assumption that $u$ vanishes on $\partial \Omega$ can be guaranteed by imposing further assumptions on $f$, as shown by Wolf \cite{MR4367723}: let $f$ be compactly supported in $\Omega$ and further satisfy $\nabla\cdot(M^{-1}f) = 0$. Then one checks easily that the solution $p$ of \eqref{eq:hpdesupc} satisfies $\nabla p = 0$, and hence $u\rvert_{\partial \Omega} = M^{-1}f =0$. A typical example given in \cite{MR4367723} is of the form $f = \rho(|x|)M\Sigma x$,  where $\rho$ is a smooth function on $\R_+$ so that $\rho(|x|)$ is compactly supported in $\Omega$, and $\Sigma$ is a constant anti-symmetric matrix.
\end{remark}

\appendix 

\section{Some analysis tools for perforated domains}
\label{app.a}

For a scalar or vector field $F$ defined in a perforated domain, namely on $\Omega^\eps$, by \emph{zero} or \emph{trivial} extension, we mean $F$ is set as zero inside the holes. The zero extension is often denoted by $\widetilde F$. Let $V_\eps \subset H^1(\Omega^\eps)$ denote the subspace
\begin{equation*}
V_\eps = \{v \in H^1(\Omega^\eps) \,:\, v\rvert_{\Omega \cap \partial \Omega^\eps} = 0\}.
\end{equation*}
It consists of functions with vanishing traces at the boundaries of the holes. A further subspace of $V_\eps$ is $H^1_0(\Omega^\eps)$. A simple but important fact is, if $F\in V_\eps$, then the extended field $\widetilde F$ belongs to $H^1(\Omega)$. 

The following version of Poincar\'e inequality plays an important role in our analysis.
\begin{proposition}[A Poincar\'e inequality]\label{prop:poincare} Let $d\ge 2$. Let $r, R$ be two positive real numbers and $r < R$. Then there exists a constant $C > 0$ that depends only on the dimension $d$, such that for any $u \in H^1(B_R(0))$ satisfying $u = 0$ in $B_r(0)$, we have
\begin{equation}
\label{eq:poincare}
\|u\|_{L^2(B_R)} \le \begin{cases}
CR(\frac{r}{R})^{-\frac{d-2}{2}} \|\nabla u\|_{L^2(B_R)}, &\qquad d\ge 3,\\
CR|\log(\frac{r}{R})|^\frac12 \|\nabla u\|_{L^2(B_R)}, &\qquad d=2.
\end{cases}
\end{equation}
\end{proposition}
\noindent It is clear that if one changes the balls to cubes in the statement, the inequality still holds with $C$ changed slightly. We refer to \cite[Lemma 3.4.1]{Allaire91-2} for the proof. 
Extend the velocity field $u^\eps$ in \eqref{eq:hetstokes} by zero inside the holes. Then inside each $\eps$-cube $Q_{z,\eps} := \eps(z+Q_1)$, we can apply \eqref{eq:poincare} with $R=\eps$ and $r=\eta\eps$. The bounding constant is then $C\sigma_\eps$ with $\sigma_\ep$ defined in \eqref{eq:sigeps}. 

Combining the resulted estimates in each $\eps$-cubes we get
\begin{equation}
\label{eq:uepspoincare}
\|u^\eps\|_{L^2(\Omega^\eps)} \le C\sigma_\eps \|\nabla u^\eps\|_{L^2(\Omega^\eps)}.
\end{equation}
First clarified in \cite{Allaire91-2}, this inequality is the key to derive energy estimates for the system \eqref{eq:hetstokes} in the dilute setting. It takes into account the various asymptotic regimes for the system depending on the relative smallness of $\eta$ with respect to $\eps$. 

As a general method to construct proper test functions for the weak formulation of \eqref{eq:hetstokes}, we introduce the restriction operator of vector fields and the associated pressure fields:
\begin{equation*}
\cR : H^1(\Omega) \to V_\eps,  \qquad \cE : L^2_0(\Omega^\eps) \to L^2_0(\Omega).
\end{equation*}
They satisfy the important duality relation 
\begin{equation}
	\label{eq:REduality}
\int_\Omega (\mathcal{E} p)\; \nabla\cdot u = \int_{\Omega^\eps} p\; \nabla\cdot (\mathcal{R} u), \qquad \forall u \in H^1_0(\Omega), \; \forall p\in L^2_0(\Omega^\eps).
\end{equation}
Such restriction operators were first constructed by Tartar in the case of $\eta = 1$, and generalized by Allaire to the dilute setting. We collect those results below, and for proofs and detailed explanations we refer to \cite[Lemma 2.2.3]{Allaire91-1}.

\begin{proposition}[Restriction operator in the unit cell]\label{prop:restr} For each fixed $\eta \in (0,1)$, there exists a linear operator $\mathcal{R}$ that maps $u\in H^1(Q_1)$ to $\mathcal{R} u \in H^1(Q_1)$, with $\mathcal{R}u = 0$ in $\eta \ol T$ and $\mathcal{R}u = u$ in $Q_1\setminus B_{1/2}$. Moreover, the operator satisfies
\begin{itemize}
	\item[(1)] If $u \in H^1(Q_1)$ and $u = 0$ on $\partial T$, then $\mathcal {R} u = u$ on $Q_1 \setminus (\eta T)$.
	\item[(2)] If $\nabla\cdot u = 0$ in $Q_1$, then $\nabla\cdot (\mathcal {R} u) = 0$ in $Q_1 \setminus (\eta \ol T)$. 
\item[(3)] There exists a constant $C$, independent of $\eta$ or $u$, such that for all $u\in H^1(Q_1)$, the following holds:
\begin{equation}
\|\nabla (\mathcal{R} u)\|_{L^2(Q_1\setminus \eta\ol{T})} \le C\left[\|\nabla u\|_{L^2(Q_1)} + \kappa_\eta \|u\|_{L^2(Q_1)}\right].
\end{equation}
where $\kappa_\eta$ is defined in \eqref{eq:sigeps}
\end{itemize}
\end{proposition}

We then apply this construction with proper scaling in each $\eps$-cube around the holes in $\Omega^\eps$, and get the following construction result.
\begin{corollary}[Restriction operator on the perforated domain]\label{coro:restr}
For each pair of $\eps,\eta \in (0,1)$, let the domain $\Omega^\eps = \Omega^{\eps,\eta}$ be as in Assumption \ref{assump:geo}. Then there is a linear operator, still denoted by $\mathcal{R}$, which maps $u \in H^1_0(\Omega)$ to $\mathcal{R} u \in H^1_0(\Omega)$, with $\mathcal{R} u = 0$ in the holes $\Omega\setminus \Omega^\eps$. Moreover, $\mathcal{R}$ satisfies:
\begin{itemize}
  \item[(1)] If $u\rvert_{\Omega^\eps} \in H^1_0(\Omega^\eps)$, then $\mathcal{R} u$ is the zero extension of $u \rvert_{\Omega^\eps}$ on $\Omega\setminus \Omega^\eps$.
	\item[(2)] If $\nabla \cdot u = 0$ in $\Omega$, then $\nabla \cdot (\mathcal{R} u) = 0$ in $\Omega^\eps$.
\item[(3)] there is a constant $C$ independent of $\eta, \eps$ or $u$, such that
\begin{equation}
\label{eq:restrbdd}
\|\nabla (\mathcal{R} u) \|_{L^2(\Omega^\eps)} \le C\left[\|\nabla u\|_{L^2(\Omega)} + \frac{1}{\sigma_\eps} \|u\|_{L^2(\Omega)}\right].
\end{equation}
\end{itemize}
\end{corollary}

The extension operator $\cE$ is defined by the duality relation \eqref{eq:REduality}, again by Tartar. Up to an additive constant over $\Omega$, $\mathcal{E}p$ is defined by
\begin{equation}
	\mathcal{E} p = \begin{cases} p, \qquad &x\in \Omega^\eps,\\
		\frac{1}{|Q_{z,\eps}|} \int_{\eps(z+Y_{\rm f})} p, \qquad &x\in \eps(z+\eta T), z\in \mathcal{I}_\eps. 
	\end{cases}
	\label{eq:cEdef}
\end{equation}


Using classical energy estimates and the Poincar\'e inequality \eqref{eq:poincare}, we get the following energy estimate: there exists a constant $C$ that is independent of $\eps$, $\eta$ and $f$ so that 
\begin{equation}
\label{eq:uunif}
\|\nabla u^\eps\|_{L^2(\Omega)} \le C(1\wedge \sigma_\eps)\|f\|_{L^2(\Omega)}, \qquad \|u^\eps\|_{L^2(\Omega)} \le C(1\wedge \sigma_\eps^2)\|f\|_{L^2(\Omega)}.
\end{equation}
For the pressure extension defined by \eqref{eq:REduality} and \eqref{eq:cEdef}, and by the properties of the restriction operator \eqref{eq:restrbdd}, we can derive the following estimate: there exists a constant $C$, depending only on $\Omega$ and $T$ such that
\begin{equation}
\label{eq:punif}
\|\mathcal{E} p^\eps\|_{L^2(\Omega)} \le C\|f\|_{L^2(\Omega)}.
\end{equation}
In view of the formula \eqref{eq:cEdef} that holds up to an additive constant, we see that the above estimate still hold if $\mathcal{E}p^\eps$ is replaced by $p^\eps$.

\section{Basic energy estimates}
\label{app.b}

We first record some basic energy estimate for Stokes systems, which will be used systematically in our convergence rates analysis.

\begin{proposition}
 \label{prop:zteq_supc}
 Suppose that Assumption \ref{assump:geo} holds, and $(v,p)$ solves
\begin{equation}
  \label{eq:vpsystem}
  \left\{
	\begin{aligned}
	  &-\sigma^2 \Delta v + \nabla p = b + \sigma (\nabla\cdot F + s), &\qquad \text{in } \Omega^\eps,\\
	&\nabla \cdot v = g, &\qquad \qquad \text{in } \Omega^\eps,\\
	&v = \mathbf{1}_{\partial \Omega}(x)h(x), &\qquad \text{in } \partial \Omega^\eps,
	\end{aligned}
	\right.
\end{equation}
where $\sigma >0$ is a real number, $b\in L^2(\Omega)$ is a vector field, $F\in L^2(\Omega)$ is a matrix field, $g \in L^2(\Omega)$ is a scalar field and $s\in H^{-1}(\Omega^\eps)$. Moreover, $v$ and $p$ are extended by zero in the holes of $\Omega^\eps$; $p$ is mean zero in $\Omega$; $g$ and $h$ satisfy the compatibility assumption 
\begin{equation}
\label{eq:ghcompatible}
  \int_{\Omega} g = \int_{\partial \Omega} N \cdot h.
\end{equation}
Then there exists a universal constant $C>0$ such that, for any $H \in H^1(\Omega)$ ($H$ is a lifting of the boundary data $h$ on $\partial\Omega$) that satisfies $H+v \in H^1_0(\Omega)$ and $H+v = 0$ in the holes of $\Omega^\eps$, we have  
  \begin{equation}
	\begin{aligned}
	  \sigma \|\nabla v\|_{L^2(\Omega)} + \|v\|_{L^2(\Omega)} + \|p\|_{L^2(\Omega^\eps)} \le & C\left\{\|b\|_{L^2(\Omega^\eps)} + \|F\|_{L^2(\Omega^\eps)} + \|s\|_{H^{-1}} + \|g\|_{L^2} \right.\\
	&\qquad  + \left. \|\nabla \cdot H\|_{L^2} + \sigma \|\nabla H\|_{L^2}\right\}
  \end{aligned}
	\label{eq:benergy}
  \end{equation}
\end{proposition}

Note that by the assumptions $v$ and $p$ are uniquely determined.

\begin{proof}\ \\
\noindent\underline{Step 1: estimates for the pressure field.}\\
Let $\hat p$ be the extension of $p$ and let $w \in H^1_0(\Omega)$ be a vector field so that $\nabla\cdot w = \hat p$. Let $\cR w$ be the restricted field, then
  \begin{equation*}
	\|\cR w\|_{L^2(\Omega)}  + \sigma \|\nabla \cR w\|_{L^2(\Omega)} \le C\|w\|_{H^1(\Omega)} \le C\|\hat p\|_{L^2}.
  \end{equation*}
  Test $\cR w$ against the equation of $(v,p)$, we get
\begin{equation*}
  \begin{aligned}
	\|\hat p\|_{L^2}^2 &= \sigma^2 \int \nabla v:\nabla \cR w - \int b\cdot \cR w + \sigma \int F:\nabla \cR w - \sigma \langle s,\cR w\rangle\\
	&\le \left(\sigma \|\nabla v\|_{L^2} + \|b\|_{L^2} + \|F\|_{L^2} + \|s\|_{H^{-1}} \right) (\sigma \|\nabla \cR w\|_{L^2}).
\end{aligned}
\end{equation*}
We can hence bound the pressure field by
\begin{equation*}
  \|\hat p\|_{L^2} \le  C\left(\sigma \|\nabla v\|_{L^2} + \|b\|_{L^2} + \|F\|_{L^2} + \|s\|_{H^{-1}} \right).
\end{equation*}

\smallskip

\noindent\underline{Step 2: estimates for the velocity field.}\\
By assumption $H+v\in H^1_0(\Omega^\eps)$. Test it against the equations of $(v,p)$, we obtain
\begin{equation*}
  \sigma^2 \int_{\Omega^\eps} \nabla v : (\nabla v + \nabla H) = \int_{\Omega^\eps} p \nabla\cdot (v+H) +  b\cdot (v + H) - \sigma F:(\nabla v + \nabla H)\, dx + \sigma \langle s, v+H\rangle.
\end{equation*}
From this we easily get
\begin{equation*}
  \begin{aligned}
  \sigma^2 \|\nabla v\|_{L^2(\Omega)}^2 &\le \sigma^2 \|\nabla v\|_{L^2}\|\nabla H\|_{L^2} + \|\hat p\|_{L^2}(\|g\|_{L^2} + \|\nabla\cdot H\|_{L^2}) \\
  &\qquad \qquad + (\|b\|_{L^2} + \|F\|_{L^2} + \|s\|_{H^{-1}})(\sigma \|\nabla v\|_{L^2} + \sigma\|\nabla H\|_{L^2}).
\end{aligned}
\end{equation*}
The desired estimate \eqref{eq:benergy} then follows immediately.
\end{proof}

\section{The unified framework for qualitative convergence}
\label{sec.app.qual}

For the sake of completeness, we recall here the qualitative results from the paper \cite{MR4145838}. The idea is to combine the classical oscillating test function method (see e.g. \cite{Tartar}) with the asymptotic analysis of the two-scale cell problems to prove qualitative homogenization results.

The basic energy estimates \eqref{eq:uunif} and \eqref{eq:punif} suggest that, as $\eps \to 0$, $u^\eps/(1\wedge \sigma^2_\eps)$ and $p^\eps$ converge (at least weakly). The limits are the homogenized (effective) model for \eqref{eq:hetstokes}. As mentioned earlier, the effective equations can be derived by the informal two-scale expansion method. In the dilute setting, the following qualitative convergence result is due to Allaire \cite{Allaire91-1,Allaire91-2}.

\begin{theorem}
	\label{thm:homqual} Under Assumption \ref{assump:geo}, let $M$ be the symmetric positive definite matrix defined in \eqref{eq:Mdef}. Then the following holds. 
\begin{itemize}
\item[(1)] In the \emph{\underline{dilute super-critical}} setting, i.e., $\sigma_\eps \to 0$ as $\eps \to 0$,  the sequence $\frac{u^\eps}{\sigma^2_\eps}$ converges weakly in $L^2(\Omega)$ to $u$, and $\mathcal{E}p^\eps$ converges strongly to $p$ in $L^2_0(\Omega)$.
\item[(2)] In the \emph{\underline{critical setting}}, i.e., $\sigma_\eps \to \sigma_0$ as $\eps \to 0$ with $\sigma_0 \in (0,\infty)$, the sequence $u^\eps$ converges weakly in $H^1_0(\Omega)$ and strongly in $L^2(\Omega)$ to $u$, and $\mathcal{E}p^\eps$ converges weakly in $L^2_0(\Omega)$ to $p$.
\item[(3)] In the \emph{\underline{sub-critical}} setting, i.e., $\sigma_\eps \to \infty$ as $\eps \to 0$, the sequence $u^\eps$ converges weakly in $H^1_0(\Omega)$ and strongly in $L^2(\Omega)$ to $u$, and $\mathcal{E}p^\eps$ converges strongly in $L^2_0(\Omega)$ to $p$.
\end{itemize}
\end{theorem}

For each $\eps \in (0,1)$, let $(u^\eps,p^\eps) \in H^1_0(\Omega^\eps) \times L^2_0(\Omega^\eps)$ be the unique solution to \eqref{eq:hetstokes}, which is extended by zero in the holes. Let $\hat p^\eps \in L^2_0(\Omega)$ stand for the extension $\mathcal{E} p^\eps$. We have seen, see \eqref{eq:uunif} and \eqref{eq:punif}, that $(u^\eps,\hat p^\eps)$ satisfy the estimates
\begin{equation*}
\|\nabla u^\eps\|_{L^2(\Omega)} + (\sigma_\eps \wedge 1)^{-1}\|u^\eps\|_{L^2(\Omega)} \le C(\sigma_\eps \wedge 1)\|f\|_{L^2(\Omega)}, \qquad \|\hat p^\eps\|_{L^2(\Omega)} \le C\|f\|_{L^2(\Omega)}.
\end{equation*}
As a result, in the super-critical setting, we can find a subsequence still denoted by $\eps \to 0$, so that $u^\eps/\sigma^2_\eps$ converges weakly in $L^2(\Omega)$ to some $u \in L^2(\Omega)$ and $\hat p^\eps$ converges weakly in $L^2_0(\Omega)$ to $p$. On the other hand, in the critical and sub-critical setting, in view of the uniform boundedness of $\|\hat u^\eps\|_{H^1}$, we can find a subsequence still denoted by $\eps \to 0$, along which $u^\eps$ converges weakly in $H^1_0(\Omega)$ and strongly in $L^2(\Omega)$ to $u$, and $\hat p^\eps$ converges weakly to $p$. To confirm that the whole sequence converges, it suffices to show, in each of the three regimes, that the possible limit $(u,p)$ is uniquely determined. 

In the following, we prove that the limiting point is unique in a unified manner for all three cases. 
Let $\eps \to 0$ denote any subsequence along which $u^\eps/(\sigma^2_\eps\wedge 1)$ and $\hat p^\eps$ converge. We determine the limit using the standard oscillating test function method. Fix any real valued test function $\phi \in C^\infty_c(\Omega;\R)$. Note that the support of $\phi$ is away from the boundary $\partial \Omega$. Clearly, $\phi v^\eps_k$ is an element of $H^1_0(\Omega^\eps)$, for each $k = 1,\dots,d$. Testing equation \eqref{eq:hetstokes} for $u^\eps$ against $\phi^\eps_k$, we get
\begin{equation*}
\int_\Omega \nabla u^\eps : (\nabla v^\eps_k) \phi + (\nabla\phi \cdot \nabla u^\eps)\cdot v^\eps_k - \hat p^\eps v^\eps \cdot \nabla \phi = \int_\Omega f \cdot v^\eps_k \phi.
\end{equation*}
Similarly, considering the equation \eqref{eq:rscell} for $v^\eps_k$ in $\eps \R^d_{\rm f}$, and using $u^\eps \phi$ as a test function, we obtain:
\begin{equation*}
\int_\Omega \nabla v^\eps_k : (\nabla  u^\eps) \phi + (\nabla\phi \cdot \nabla v^\eps_k)\cdot  u^\eps - q^\eps_k  u^\eps \cdot \nabla \phi = \int_\Omega \frac{ u^\eps}{\sigma_\eps^2}\cdot e_k  \phi.
\end{equation*}
Subtracting the two equations, we get:
\begin{equation}
\label{eq:qualweak}
\int_\Omega (\nabla\phi \cdot \nabla  u^\eps)\cdot v^\eps_k - (\nabla\phi \cdot \nabla v^\eps_k)\cdot  u^\eps + q^\eps_k  u^\eps\cdot \nabla \phi - \hat p^\eps  v^\eps\cdot \nabla \phi = \int_\Omega f \cdot v^\eps_k \phi  -\frac{ u^\eps}{\sigma_\eps^2}\cdot e_k \phi.
\end{equation}
We need to send $\eps \to 0$ in the equation above. An important consequence of the subtraction above is, the resulting integrals above all involve products of a weakly converging function with a strongly converging one, and hence the limit can be addressed. There are four terms in the integrand on the left hand side and two terms on the right. We label the integrals of each of those terms by $I_1,I_2,\dots,I_6$, according to their order of appearance from left to right in \eqref{eq:qualweak}.

\subsection{Super-critical size ratio} In this setting, $\sigma_\eps \to 0$, or equivalently, $\eta \gg \eta_*$. As the setting of $\eta = O(1)$ corresponds to the classical setting already treated by Tartar, we only consider the case of vanishing volume fraction, i.e. $\eta \to 0$ as $\eps \to 0$.

By \eqref{eq:uunif} and \eqref{eq:punif}, there exists a subsequence still denoted by $\eps \to 0$, such that $ u^\eps/\sigma_\eps^2$ converges to $u$ weakly in $L^2(\Omega)$, and $\mathcal{E}p^\eps$ converges to $p$ weakly in $L^2_0(\Omega)$. Passing to the limit $\eps \to 0$ in equation \eqref{eq:qualweak}, and using Theorem \ref{lem:cellqual} and Lemma \ref{lem:vq}, we can show
\begin{equation*}
\begin{aligned}
&|I_1| \le \|\nabla \phi\|_{L^d}\|\nabla  u^\eps\|_{L^2}\|v^\eps\|_{L^{2^*}} \le C\sigma_\eps.\\
&|I_2| \le \|\nabla \phi\|_{L^\infty}\| u^\eps\|_{L^2}\|\nabla v^\eps\|_{L^2} \le C\sigma_\eps.\\
&|I_3| \le \|q^\eps_k\|_{L^2}\| u^\eps\|_{L^2}\|\nabla \phi\|_{L^\infty} \le C\sigma_\eps.
\end{aligned}
\end{equation*}
Those terms hence vanishe in the limit. For $I_4$, since $\mathcal{E}p^\eps$ converges weakly along the subsequence and $v^\eps_k$ converges strongly, we get
\begin{equation*}
\lim_{\eps \to 0} I_4 = - \int_\Omega p (M^{-1}e_k)\cdot \nabla \phi.
\end{equation*}
The limit of $I_5 + I_6$ is clear, and we obtain (using in particular the fact that $M^{-1}$ is symmetric and positive definite)
\begin{equation*}
\begin{aligned}
0 = &\int p(M^{-1}e_k)\cdot \nabla \phi + \left(f\cdot (M^{-1}e_k) - u\cdot e_k \right)\phi \\
= &\int \left(f\cdot (M^{-1}e_k) - u\cdot e_k - M^{-1}\nabla p \cdot e_k\right)\phi, \qquad \forall \varphi \in C^\infty_c(\Omega).
\end{aligned}
\end{equation*}
Due to the arbitrariness of $\phi$, we deduce that $(u,p)$ solves \eqref{eq:hpdesupc}. 
This shows that $p \in H^1(\Omega)$. On the other hand, for any scalar valued function $\psi \in H^1(\Omega)$ and for the converging sequence $u^\eps$, we have
\begin{equation*}
\int_\Omega \frac{ u^\eps}{\sigma^2_\eps} \cdot \nabla \psi = 0.
\end{equation*}
Sending $\eps \to 0$, We hence get
\begin{equation*}
\int_\Omega M^{-1}(f-\nabla p) \cdot \nabla \psi =0, \qquad \forall \psi \in H^1(\Omega).
\end{equation*}
This is precisely the weak formulation for the following problem for $p \in H^1(\Omega) \cap L^2_0(\Omega)$:
\begin{equation*}
\left\{\begin{aligned}
&\nabla \cdot [M^{-1}(f-\nabla p)] = 0 \qquad &\text{in } \Omega,\\
&n \cdot M^{-1}(f-\nabla p) = 0 \qquad &\text{on } \partial \Omega.
\end{aligned}
\right.
\end{equation*}
Note that the above problem is a rephrasing of \eqref{eq:hpdesupc}, and it admits a unique solution. As a result, the possible limit for any converging subsequence of $(u^\eps,p^\eps)$ is uniquely determined. Hence, the whole sequence converges and this establishes the qualitative convergence result in the super-critical setting.

\subsection{Critical size ratio} In this setting, $\sigma_\eps$ is of order one and $\sigma_\eps \to \sigma_0$ for some positive number $\sigma_0$. By \eqref{eq:uunif} and \eqref{eq:punif}, the sequence $ u^\eps$ is uniformly bounded in $H^1_0(\Omega)$ and $\hat p^\eps$ is uniformly bounded in $L^2_0(\Omega)$. Hence, there exists a subsequence, still denoted by $\eps \to 0$, such that $ u^\eps$ converges weakly in $H^1_0(\Omega)$ and strongly in $L^2(\Omega)$ to some vector field $u \in H^1_0(\Omega)$, and $\hat p^\eps$ converges weakly in $L^2_0(\Omega)$ to some scalar field $p$.

Examine \eqref{eq:qualweak} along this subsequence. Because $\nabla  u^\eps$ converges weakly and $v^\eps_k$ converges strongly, we have
\begin{equation*}
\lim_{\eps \to 0} I_1 = \int_\Omega (\nabla \phi \cdot \nabla u) \cdot (M^{-1}e_k).
\end{equation*}
For the same reason $I_2$ converges and its limit is zero due to \eqref{eq:gradvconv}. The limits of $I_4$ and $I_5 + I_6$ are also clear. For $I_3$, we use \eqref{eq:qquant1} and the uniform bound on $\| u^\eps\|_{H^1}$ to conclude that, for $d\ge 3$, 
\begin{equation*}
|I_3| \le C\eta^{\frac{d-2}{2}} \| u^\eps \cdot \nabla \phi\|_{H^1} \rightarrow 0, \quad \text{as } \eps \to 0.
\end{equation*}
The same limit also holds for $d=2$. We hence get, along the chosen subsequence, 
\begin{equation*}
\int_{\Omega} (M^{-1}e_k) \cdot \left[\nabla \phi \cdot \nabla u  - p\nabla \phi\right] = \int_\Omega M^{-1}e_k \cdot (\phi f) - e_k \cdot \frac{u}{\sigma^2_0} \phi, \qquad \forall\, \phi \in C^\infty_c(\Omega).
\end{equation*}
After multiplying  $M$ from the left on both sides of the equality and integration by parts, we obtain
\begin{equation*}
\int_\Omega \nabla u \cdot \nabla \phi - p\nabla \phi + \frac{Mu}{\sigma^2_0} \phi = \int_\Omega f\phi, \qquad \forall \varphi \in C^\infty_c(\Omega).
\end{equation*}
This shows that $(u,p)$ solves the problem \eqref{eq:hpdec} in the distributional sense. Since $u\in H^1_0(\Omega)$ and $p\in L^2(\Omega)$, the pair $(u,p)$ also solve the above problem in the weak sense. On the other hand, it is clear from the positivity of $M$ that the weak solution is unique. We conclude that the limiting point of $( u^\eps,\hat p^\eps)$ is uniquely determined, and the whole sequence converges accordingly.

\subsection{Sub-critical size holes} In this setting, $\sigma_\eps \to \infty$ and $\sigma_\eps^{-1}$ vanishes in the limit of $\eps \to 0$. In view of \eqref{eq:uunif}, there exist subsequences along which $u^\eps$ converges weakly in $H^1$ and strongly in $L^2$ to some vector field $u \in H^1_0(\Omega)$, and $\hat p^\eps$ converges weakly in $L^2$ to some scalar field $p$. 

The argument that we use to characterize the possible limiting points $(u,p)$ is essentially the same as in the critical case, except that $I_3$ is simpler to deal with. Indeed, it vanishes in the limit in view of \eqref{eq:qquant}. After sending $\eps \to 0$ along the converging subsequence, \eqref{eq:qualweak} becomes
\begin{equation*}
\int_{\Omega} (M^{-1}e_k) \cdot \left[\nabla \phi \cdot \nabla u  - p\nabla \phi\right] = \int_\Omega M^{-1}e_k \cdot (\phi f), \qquad \forall\, \phi \in C^\infty_c(\Omega).
\end{equation*}
Arguing as in the critical setting, we conclude that any limit $(u,p)$ of $( u^\eps, \hat p^\eps)$ must be given by the unique solution of \eqref{eq:hpdesubc}. The whole sequence of $( u^\eps, \hat p^\eps)$ hence converges to the unique solution of \eqref{eq:hpdesubc}. This completes the proof of Theorem \ref{thm:homqual}.

\section{The unified framework for quantitative convergence}
\label{sec.app.quant}

\subsection{Critical holes}

We modify \eqref{eq:zetatau-super} and consider the following discrepancy functions:
\begin{equation}
\label{eq:errsc}
\zeta^\eps = u^\eps -  \frac{\sigma^2_\eps}{\sigma_0^2} v^\eps_k(x) (Mu)^k \qquad \text{and} \qquad \tau^\eps = p^\eps(x) - p(x) - \frac{\sigma_\eps^2}{\sigma_0^2} q^\eps_k(x) (Mu)^k.
\end{equation}
Direct computation then shows that the following equations hold in the perforated domain $\Omega^\eps$:
\begin{equation}
  \label{eq:zteq_c}
\left\{
\begin{aligned}
  &-\Delta \zeta^\eps + \nabla \tau^\eps = \left(({\sigma_\eps}/{\sigma_0})^2-1\right)\Delta u  + 2({\sigma_\eps}/{\sigma_0})^2 \partial_\ell[(v^\eps_k-M^{-1}e_k)\partial_\ell (Mu)^k] +({\sigma_\eps}/{\sigma_0})^2 \check{s}\\ 
  & \qquad\qquad\qquad - ({\sigma_\eps}/{\sigma_0})^2 (v^\eps_k - M^{-1}e_k)^i\Delta(Mu)^k
 - ({\sigma_\eps}/{\sigma_0})^2 q^\eps_k \nabla (Mu)^k,\\
&\nabla \cdot \zeta^\eps = -({\sigma_\eps}/{\sigma_0})^2(v^\eps_k - M^{-1}e_k) \cdot \nabla (Mu)^k. 
\end{aligned}
\right.
\end{equation}
Moreover, since the solution $u$ of the homogenized problem \eqref{eq:hpdec} vanishes on $\partial \Omega$, the boundary condition $\zeta^\eps = 0$ holds on $\partial \Omega^\eps$. We recognize the system as a special case of \eqref{eq:vpsystem} with the data: $\sigma =1$, $h=0$ (since $u=0$), $s = s_1 + s_2$, and
\begin{equation*}
  \begin{aligned}
	&g=-(\sigma_\eps/\sigma_0)^2(v^\eps_k-M^{-1}e_k)\cdot \nabla(Mu)^k, \quad  F = 2(\sigma_\eps/\sigma_0)^2 (v^\eps_k-M^{-1}e_k)\nabla(Mu)^k,\\
	&b = [(\sigma_\eps/\sigma_0)^2-1]\Delta u - (\sigma_\eps/\sigma_0)^2\Delta(Mu)^k (v^\eps_k-M^{-1}e_k),	 \quad s_1 = (\sigma_\eps/\sigma_0)^2 \check{s},
  \end{aligned}
\end{equation*}
and
\begin{equation*}
  \langle s_2,\varphi\rangle = -(\sigma_\eps/\sigma_0)^2 \int q^\eps_k \nabla(Mu)^k \cdot \varphi.
\end{equation*}

\begin{proof}[Proof of Theorem \ref{thm:homerr} in the critical setting] 
  Note that in the critical setting, we have $\sigma_\eps \sim \sigma_0 \sim 1$ and $\kappa_\eta \sim \eps$. By the estimates of $(v^\eps_k,q^\eps_k)$'s, we easily get
  \begin{equation*}
	\|g\|_{L^2} + \|F\|_{L^2} \le C\kappa_\eta \|\nabla u\|_{L^\infty}, \quad \|b\|_{L^2} \le C(\kappa_\eta + |1-(\sigma_\eps/\sigma_0)^2|)\|\nabla^2 u\|_{L^\infty}.
  \end{equation*}
  For the distributions $s_1$ and $s_2$, by \eqref{eq:sepskbdd-2} and following the analysis in the previous setting, and noting that $\kappa_\eta \ll \sqrt{\eps}$, we check
  \begin{equation*}
	\|s_1\|_{H^{-1}(\Omega^\eps)} \le C\sqrt{\eps}\|u\|_{W^{1,\infty}}.
  \end{equation*}
  For $s_2$, we use \eqref{eq:qquant1} to get $\|s_2\|_{H^{-1}(\Omega^\eps)} \le C\kappa_\eta \|u\|_{W^{1,\infty}}$.

  Combine all the estimates above and apply \eqref{eq:benergy}. We conclude that
\begin{equation*}
  \|\zeta^\eps\|_{H^1(\Omega)} + \|\hat\tau^\eps\|_{L^2(\Omega)} \le C\left(\kappa_\eta + |1-(\sigma_\eps/\sigma_0)^2| + \sqrt{\eps}\right)\|u\|_{W^{2,\infty}}.
\end{equation*}
Note that $\sqrt{\eps}$ is much larger than $\kappa_\eta$ since $\kappa_\eta \sim \eps$, so the $\check{s}$ term makes the largest contribution to the error. The above establishes the first part of \eqref{eq:err_c}. The second part follows from the estimate of $\hat\tau^\eps$ and the fact that $\int \tau^\eps$ is of order $O(\kappa_\eta)$, which in turns follows from \eqref{eq:qquant1}.
\end{proof}

\begin{remark}
Note that the choice \eqref{eq:errsc} corresponds to setting the following correctors
\begin{equation*}
  r_\eps = (\sigma_\eps/\sigma_0)^2 (v^\eps_k-M^{-1}e_k)(Mu)^k + [(\sigma_\eps/\sigma_0)^2 -1]u, \quad t_\eps = (\sigma_\eps/\sigma_0)^2 q^\eps_k (Mu)^k.
\end{equation*}
Clearly we have
\begin{equation*}
  \|r_\eps\|_{L^2} \le C(\kappa_\eta +|1-(\sigma_\eps/\sigma_0)^2|)\|u\|_{L^\infty}. 
\end{equation*}
This allows us to show smallness of $\|u^\eps-u\|_{L^2}$. Note, however, the $L^2$ norm of the pressure corrector does not seem small.
\end{remark}

\subsection{Sub-critical holes}

We modify \eqref{eq:zetatau-super} and consider the discrepancy functions:
\begin{equation}
\label{eq:errsubc}
\zeta^\eps = u^\eps -  v^\eps_k(x) (Mu)^k \qquad \text{and} \qquad \tau^\eps = p^\eps(x) - p(x) - q^\eps_k(x) (Mu)^k.
\end{equation}
Direct computation shows that
\begin{equation}
  \label{eq:vteq_subc}
\left\{
\begin{aligned}
	&-\Delta \zeta^\eps + \nabla \tau^\eps = -\sigma^{-2}_\eps Mu  + \check{s} - (v^\eps_k-M^{-1}e_k)\Delta (Mu)^k\\
&\qquad\qquad\qquad\qquad + 2\partial_\ell[(v^\eps_k-M^{-1}e_k)^i \partial_\ell(Mu)^k] 
 - q^\eps_k \nabla (Mu)^k,\\
 &\nabla \cdot \zeta^\eps = -(v^\eps_k - M^{-1}e_k) \cdot \nabla (Mu)^k.
 \end{aligned}
 \right.
\end{equation}
Again the boundary condition $\zeta^\eps = 0$ on $\partial \Omega^\eps$ is satisfied.

We recognize the above system for $(\zeta^\eps,\hat\tau^\eps)$, where $\hat\tau^\eps$ denotes the mean zero part of $\tau^\eps$, as a special case of \eqref{eq:vpsystem} with the following data: $\sigma=1$, $h=0$, $s=s_1+s_2$, and
\begin{equation*}
  \begin{aligned}
	&b = -\sigma^{-2}_\eps Mu - (v^\eps_k -M^{-1}e_k)\Delta(Mu)^k, \quad &&F = 2(v^\eps_k-M^{-1}e_k)\otimes \nabla(Mu)^k,\\
	&g = -(v^\eps_k-M^{-1}e_k)\cdot \nabla(Mu)^k, \quad &&s_1 = \check{s}, \quad \langle s_2,\varphi\rangle = -\int q^\eps_k \nabla(Mu)^k\cdot \varphi.
  \end{aligned}
\end{equation*}

\begin{proof}[Proof of Theorem \ref{thm:homerr} in the sub-critical setting] From our estimates of $v^\eps_k$ and $q^\eps_k$'s, we obtain
  \begin{equation*}
	\begin{aligned}
	&\|F\|_{L^2} + \|g\|_{L^2} \le C\kappa_\eta \|u\|_{W^{1,\infty}}, \quad \|s_2\|_{H^{-1}(\Omega^\eps)} \le C\kappa_\eta \|u\|_{W^{2,\infty}},\\
	&\|b\|_{L^2} \le C(\sigma_\eps^{-2}\|u\|_{L^\infty} + \kappa_\eta\|\Delta u\|_{L^\infty}).
  \end{aligned}
  \end{equation*}
  The estimate of $s_1$ is still a consequence of \eqref{eq:sepskbdd-2} and reads
  \begin{equation*}
	\|s_1\|_{H^{-1}(\Omega^\eps)} \le C\sigma_\eps^{-1}(\kappa_\eta + \sqrt{\eps}) \|u\|_{W^{1,\infty}}.
  \end{equation*}
Note that in the sub-critical setting $O(\kappa_\eps) \ll O(\eps)$, and $O(\kappa_\eps/\sigma_\eps) \ll O(\sigma_\eps^{-2})$. All the results above combined yield
\begin{equation*}
  \|\nabla \zeta^\eps\|_{L^2(\Omega)} + \|\hat\tau^\eps\|_{L^2(\Omega)} \le C\left(\frac{1}{\sigma_\eps^2} + \frac{\sqrt{\eps}}{\sigma_\eps}\right)\|u\|_{W^{2,\infty}}
\end{equation*}
The above establishes the first part of \eqref{eq:err_subc}. The second part also holds because the mean of $\tau^\eps$ is smaller, of order $O(\kappa_\eta)$. 
\end{proof}

\begin{remark} Note that \eqref{eq:errsubc} corresponds to choosing the correctors
\begin{equation*}
  r_\eps = (v^\eps_k-M^{-1}e_k)(Mu)^k, \qquad t_\eps = q^\eps_k (Mu)^k.
\end{equation*}
Clearly, $\|r_\eps\|\le C\kappa_\eta\|u\|_{L^\infty}$ and $\|t_\eps\|_{L^2} \le C\sigma_\eps^{-1}\|u\|_{L^\infty}$. The second part of \eqref{eq:err_subc} then follows.
\end{remark}

\section*{Acknowledgment}
WJ is partially supported by the New Cornerstone Science Foundation grant - NCI202310 and by MOST-2023YFA1008902. YL has been supported by the Recruitment Program of Global Experts of China. YL is partially supported by the  NSF of China under Grant 12171235. CP is partially supported by the Agence Nationale de la Recherche, project SINGFLOWS grant ANR-18-CE40-0027-01, project CRISIS grant ANR-20-CE40-0020-01, project BOURGEONS grant ANR-23-CE40-0014-01 and by the CY Initiative of Excellence, project CYFI (CYngular Fluids and Interfaces) and project CYNA (CY Nonlinear Analysis).

\section*{Data availability statement}

Data sharing is not applicable to this article as no datasets were generated or analyzed during the current study.

\section*{Conflict of interest}

The authors declare that they have no conflict of interest.

\bibliographystyle{abbrv}
\bibliography{stokes.bib}

\begin{thebibliography}{10}

\bibitem{Allaire91-1}
G.~Allaire.
\newblock Homogenization of the {N}avier-{S}tokes equations in open sets
  perforated with tiny holes. {I}. {A}bstract framework, a volume distribution
  of holes.
\newblock {\em Arch. Rational Mech. Anal.}, 113(3):209--259, 1990.

\bibitem{Allaire91-2}
G.~Allaire.
\newblock Homogenization of the {N}avier-{S}tokes equations in open sets
  perforated with tiny holes. {II}. {N}oncritical sizes of the holes for a
  volume distribution and a surface distribution of holes.
\newblock {\em Arch. Rational Mech. Anal.}, 113(3):261--298, 1990.

\bibitem{Allaire-3}
G.~Allaire.
\newblock Continuity of the {D}arcy's law in the low-volume fraction limit.
\newblock {\em Ann. Scuola Norm. Sup. Pisa Cl. Sci. (4)}, 18(4):475--499, 1991.

\bibitem{AmmKan}
H.~Ammari and H.~Kang.
\newblock {\em Polarization and moment tensors}, volume 162 of {\em Applied
  Mathematical Sciences}.
\newblock Springer, New York, 2007.
\newblock With applications to inverse problems and effective medium theory.

\bibitem{BAO24}
L.~Balazi, G.~Allaire, and P.~Omnes.
\newblock Sharp convergence rates for the homogenization of the {S}tokes
  equations in a perforated domain.
\newblock {\em HAL-04541828}, 2024.

\bibitem{CioMur-1}
D.~Cioranescu and F.~Murat.
\newblock Un terme \'{e}trange venu d'ailleurs.
\newblock In {\em Nonlinear partial differential equations and their
  applications. {C}oll\`ege de {F}rance {S}eminar, {V}ol. {II} ({P}aris,
  1979/1980)}, volume~60 of {\em Res. Notes in Math.}, pages 98--138, 389--390.
  Pitman, Boston, Mass.-London, 1982.

\bibitem{CMcIM82}
R.~R. Coifman, A.~McIntosh, and Y.~Meyer.
\newblock L'int{\'e}grale de {Cauchy} d{\'e}finit un op{\'e}ratuer borne sur
  {{\(L^ 2 \)}}pour les courbes lipschitziennes.
\newblock {\em Ann. Math. (2)}, 116:361--387, 1982.

\bibitem{FabKenVer}
E.~B. Fabes, C.~E. Kenig, and G.~C. Verchota.
\newblock The {D}irichlet problem for the {S}tokes system on {L}ipschitz
  domains.
\newblock {\em Duke Math. J.}, 57(3):769--793, 1988.

\bibitem{FJ_highcontrast}
X.~Fu and W.~Jing.
\newblock Convergence rate and uniform lipschitz estimate in periodic
  homogenization of high-contrast elliptic systems.
\newblock {\em arXiv: 2404.11396}, 2024.

\bibitem{MR4714506}
G.~Jankowiak and A.~Lozinski.
\newblock Non-conforming multiscale finite element method for {S}tokes flows in
  heterogeneous media. {P}art {II}: {E}rror estimates for periodic
  microstructure.
\newblock {\em Discrete Contin. Dyn. Syst. Ser. B}, 29(5):2298--2332, 2024.

\bibitem{Jing20}
W.~Jing.
\newblock A {U}nified {H}omogenization {A}pproach for the {D}irichlet {P}roblem
  in {P}erforated {D}omains.
\newblock {\em SIAM J. Math. Anal.}, 52(2):1192--1220, 2020.

\bibitem{MR4172687}
W.~Jing.
\newblock Layer potentials for {L}am\'{e} systems and homogenization of
  perforated elastic medium with clamped holes.
\newblock {\em Calc. Var. Partial Differential Equations}, 60(1):Paper No. 2,
  32, 2021.

\bibitem{Keller}
J.~B. Keller.
\newblock Darcy's law for flow in porous media and the two-space method.
\newblock In {\em Nonlinear partial differential equations in engineering and
  applied science ({P}roc. {C}onf., {U}niv. {R}hode {I}sland, {K}ingston,
  {R}.{I}., 1979)}, volume~54 of {\em Lecture Notes in Pure and Appl. Math.},
  pages 429--443. Dekker, New York, 1980.

\bibitem{Ladyzhenskaya}
O.~A. Ladyzhenskaya.
\newblock {\em The mathematical theory of viscous incompressible flow}.
\newblock Second English edition, revised and enlarged. Translated from the
  Russian by Richard A. Silverman and John Chu. Mathematics and its
  Applications, Vol. 2. Gordon and Breach, Science Publishers, New
  York-London-Paris, 1969.

\bibitem{MR4145838}
Y.~Lu.
\newblock Homogenization of {S}tokes equations in perforated domains: a unified
  approach.
\newblock {\em J. Math. Fluid Mech.}, 22(3):Paper No. 44, 13, 2020.

\bibitem{MM96}
E.~Marui\'{c}-Paloka and A.~Mikeli\'{c}.
\newblock An error estimate for correctors in the homogenization of the
  {S}tokes and the {N}avier-{S}tokes equations in a porous medium.
\newblock {\em Boll. Un. Mat. Ital. A (7)}, 10(3):661--671, 1996.

\bibitem{Mas04}
N.~Masmoudi.
\newblock Some uniform elliptic estimates in a porous medium.
\newblock {\em C. R. Math. Acad. Sci. Paris}, 339(12):849--854, 2004.

\bibitem{RS24}
R.~Righi and Z.~Shen.
\newblock Dirichlet problems in perforated domains.
\newblock {\em arXiv:2402.13021}, 2024.

\bibitem{Sanchez}
E.~S\'{a}nchez-Palencia.
\newblock {\em Nonhomogeneous media and vibration theory}, volume 127 of {\em
  Lecture Notes in Physics}.
\newblock Springer-Verlag, Berlin-New York, 1980.

\bibitem{MR4267502}
Z.~Shen.
\newblock Large-scale {L}ipschitz estimates for elliptic systems with periodic
  high-contrast coefficients.
\newblock {\em Comm. Partial Differential Equations}, 46(6):1027--1057, 2021.

\bibitem{Shen22}
Z.~Shen.
\newblock Compactness and large-scale regularity for {D}arcy's law.
\newblock {\em J. Math. Pures Appl. (9)}, 163:673--701, 2022.

\bibitem{MR4432951}
Z.~Shen.
\newblock Sharp convergence rates for {D}arcy's law.
\newblock {\em Comm. Partial Differential Equations}, 47(6):1098--1123, 2022.

\bibitem{Shen23-JDE}
Z.~Shen.
\newblock Homogenization of boundary value problems in perforated {L}ipschitz
  domains.
\newblock {\em J. Differential Equations}, 376:283--339, 2023.

\bibitem{Shen23}
Z.~Shen.
\newblock Uniform estimates for {D}irichlet problems in perforated domains.
\newblock {\em Vietnam J. Math.}, 51(4):845--867, 2023.

\bibitem{SW23}
Z.~Shen and J.~Wallace.
\newblock Uniform {$W^{1, p}$} estimates and large-scale regularity for
  {D}irichlet problems in perforated domains.
\newblock {\em J. Funct. Anal.}, 285(10):Paper No. 110118, 33, 2023.

\bibitem{SZ23}
Z.~Shen and J.~Zhuge.
\newblock Uniform regularity for degenerate elliptic equations in perforated
  domains.
\newblock {\em arXiv: 2311.10258}, 2023.

\bibitem{Tartar}
L.~Tartar.
\newblock Incompressible fluid flow in a porous media - convergence of the
  homogenization process.
\newblock {\em Appendix to \cite{Sanchez}}, pages 368--377, 1980.

\bibitem{WXZ22}
L.~Wang, Q.~Xu, and Z.~Zhang.
\newblock Corrector estimates and homogenization error of unsteady flow ruled
  by {D}arcy's law.
\newblock {\em arXiv:2202.04826}, 2022.

\bibitem{WXZ21}
L.~Wang, Q.~Xu, and P.~Zhao.
\newblock Convergence rates for linear elasticity systems on perforated
  domains.
\newblock {\em Calc. Var. Partial Differential Equations}, 60(2):Paper No. 74,
  51, 2021.

\bibitem{MR4367723}
S.~Wolf.
\newblock Homogenization of the {S}tokes system in a non-periodically
  perforated domain.
\newblock {\em Multiscale Model. Simul.}, 20(1):72--106, 2022.

\end{thebibliography}

\end{document}